\documentclass[10pt]{amsart}
\pdfoutput=1
\usepackage{amssymb,amscd,graphicx,enumerate,epsfig}
\usepackage[bookmarks=true]{hyperref}

\newtheorem{theorem}{Theorem}[section]

\newtheorem{proposition}[theorem]{Proposition}
\newtheorem{corollary}[theorem]{Corollary}

\newtheorem{lemma}[theorem]{Lemma}

\theoremstyle{definition}
\newtheorem{definition}[theorem]{Definition}

\newcommand{\Case}[1]{\textbf{Case #1.}}
\newcommand{\WR}{\mathcal{WR}}
\begin{document}

\title{On unstabilzed genus three critical Heegaard surfaces}
\author{Jungsoo Kim}
\date{25, Mar, 2012}

\begin{abstract}
	Let $M$ be a compact orientable irreducible $3$-manifold and $H$ be an unstabilized genus three Heegaard splitting of $M$.
	In this article, we will define a simplicial complex of weak reducing pairs for $H$ and find several properties of this complex.
	Using this method, we will prove that an unstabilized Heegaard splitting of genus three is critical in a certain condition.
	In addition, we will show that the standard genus three Heegaard splitting for $T^3$ and the induced Heegaard splitting of the three component chain exterior by a certain tunnel system are critical as examples of the main theorem.
\end{abstract}

\address{\parbox{4in}{
	Department of Mathematics, Chung-Ang University, Seoul, Korea}} 
\email{pibonazi@gmail.com}

\maketitle
\tableofcontents

\section{Introduction and result} 

Throughout this paper, all surfaces and 3-manifolds will be taken to be compact and orientable.
In \cite{Bachman1}, Bachman introduced the concept \textit{``a critical surface''} and he proved several theorems about incompressible surfaces, the number of Heegaard splittings with respect to its genus, and the minimal genus common stabilization.
Since a critical surface has disjoint compressions on it's both sides, if the surface is a Heegaard surface, then the splitting is weakly reducible, i.e. a critical Heegaard splitting is a kind of weakly reducible splitting.
But in some aspects, it shares common properties with strongly irreducible splittings. 
For example, if the splitting is strongly irreducible or critical, then the manifold is irreducible (Lemma 3.5 of \cite{Bachman2}.) 
Indeed, the intersection of an incompressible surface $S$ and a Heegaard surface $F$ can be isotoped essential on both $S$ and $F$ if the splitting is critical or strongly irreducible (see Theorem 5.1 of \cite{Bachman1} and Lemma 6 of \cite{Schultens2}.) 
Bachman also proved Gordon's conjecture by using the series of generalized Heegaard splittings and critical Heegaard splittings (see \cite{Bachman2}.)
In his recent work \cite{Bachman3}, he also introduced the concept \textit{``topologically minimal surfaces''}, where a strongly irreducible surface is an index $1$ topological minimal surface, and a critical surface is an index $2$ topological minimal surface, this is a way to regard strongly irreducible surfaces and critical surfaces in a unified viewpoint. 

Although critical Heegaard splittings have many powerful properties as proved in Bachman's recent works, it is not easy to determine whether a weakly reducible splitting is critical.
For genus two $3$-manifolds, a weakly reducible splitting is also a reducible splitting (see \cite{Thompson}, the proof can be extended easily to the case with non-empty boundary.)
Moreover, if a $3$-manifold has a reducible splitting, then the manifold is reducible or the splitting is stabilized (see \cite{SaitoScharlemannSchultens}.)
Since a reducible manifold cannot have a critical Heegaard splitting, we need to consider the manifolds of genus at least three.
In this article, we will prove Theorem \ref{theorem-main}.  

\begin{theorem} \label{theorem-main}
	Let $M$ be an orientable irreducible $3$-manifold and $H=(V,W;F)$ be an unstabilized genus three Heegaard splitting.
	Suppose that there is no weak reducing pair such that each of both disks cuts off a solid torus in its compression body.
	If we can choose two weak reducing pairs $(D_0, E_0)$ and $(D_1, E_1)$ such that $\partial D_0 \cap \partial E_1 \neq \emptyset$ and $\partial D_1 \cap \partial E_0\neq \emptyset$ up to isotopy, then $H$ is critical.
\end{theorem}

We can induce following corollary,

\begin{corollary} \label{corollary-main}
	Let $M$ be a closed orientable irreducible $3$-manifold and $H=(V,W;F)$ be an unstabilized genus three Heegaard splitting.
	Suppose that $H$ is not an amalgamation of two genus two splittings along a torus.
	If we can choose two weak reducing pairs $(D_0, E_0)$ and $(D_1, E_1)$ such that $\partial D_0 \cap \partial E_1 \neq \emptyset$ and $\partial D_1 \cap \partial E_0\neq \emptyset$ up to isotopy, then $H$ is critical.
\end{corollary}

In 2002, Moriah proved that if an orientable $3$-manifold has a weakly reducible Heegaard splitting of minimal genus, then $M$ contains an essential surface of positive genus (see \cite{Moriah}.) 
(In 1987, Casson and Gordon proved this when $M$ is closed and the Heegaard splitting is irreducible (see \cite{CassonGordon}.))
By using Bachman's result (Theorem 5.1 of \cite{Bachman1}), we can directly get the following corollary.
 
\begin{corollary} \label{corollary-main-2}
	Let $M$ be an orientable irreducible genus three $3$-manifold and $H=(V,W;F)$ be a minimal genus Heegaard splitting.
	Suppose that there is no weak reducing pair such that each of both disks cuts off a solid torus in its compression body.
	If we can choose two weak reducing pairs $(D_0, E_0)$ and $(D_1, E_1)$ such that $\partial D_0 \cap \partial E_1 \neq \emptyset$ and $\partial D_1 \cap \partial E_0\neq \emptyset$ up to isotopy, then there is an essential surface $S$ of positive genus in $M$ such that $F\cap S$ is essential on both $F$ and $S$.
\end{corollary}

In a genus three Heegaard splitting, if each of both disks for a weak reducing pair cuts off a solid torus in its compression body, then both disks are separating in their compression bodies.
Therefore, we can use Theorem \ref{theorem-main} and its corollaries when if there is no weak reducing pair whose two disks are separating in their compression bodies.

This article is organized as follows.
In section \ref{section2}, we introduce some basic notions, define the complex of weak reducing pairs $\WR$, and find basic properties of $\WR$, especially about the criticality of a Heegaard surface, the shapes of simplices in $\WR$, and stabilizations of a Heegaard surface.
In section \ref{section-proof-main-theorem}, we will induce special properties of $\WR$ when the manifold is irreducible, the splitting is unstabilzed, and the genus is three.
By using these properties, we will prove Theorem \ref{theorem-main}.
In section \ref{section-proof-corollary-main}, we will find additional properties of $\WR$ when the manifold is closed and prove Corollary \ref{corollary-main}.
In section \ref{section-examples}, we will show that the standard splitting of genus three for the three-torus $T^3$ is critical and find an incompressible surface $S'$ which satisfies Corollary \ref{corollary-main-2} as an example of Theorem \ref{theorem-main}. In addition, we will show that the induced Heegaard splitting of the three component chain exterior by a certain tunnel system is critical.

\section{Critical surfaces and the complex of weak reducing pairs\label{section2}}

A \textit{compression body} is a $3$-manifold which can be obtained by starting with some closed, orientable, connected surface $F$, forming the product $F\times I$, attaching some number of $2$-handles to $F\times\{1\}$ and capping off all resulting $2$-sphere boundary components that are not contained in $F\times\{0\}$ with $3$-balls.
The boundary component $F\times\{0\}$ is referred to as $\partial_+$. The rest of the boundary is referred to as $\partial_-$. 
A \textit{Heegaard splitting} of a $3$-manifold $M$ is an expression of $M$ as a union $V\cup_F W$, where $V$ and $W$ are compression bodies that intersect in a transversally oriented surface $F=\partial_+V=\partial_+ W$.
We will use the expression $(V,W;F)$ for a Heegaard splitting.
If $(V,W;F)$ is a Heegaard splitting of $M$ then we say that $F$ is a \textit{Heegaard surface}.
We say that the pair $(D,E)$ is a \textit{weak reducing pair} for $F$ if $D\subset V$ and $E\subset W$ are disjoint compressing disks.
A Heegaard surface is \textit{strongly irreducible} if it is compressible to both sides but has no weak reducing pairs.
From now, we will use the letter ``D'' for compressing disks in $V$, ``E'' for compressing disks in $W$, ``F'' for the Heegaard surface in the given Heegaard splitting, and ``H'' for the name of a Heegaard splitting. 

\begin{definition}[D. Bachman, Definition 3.3 of \cite{Bachman2}]\label{definition-critical}
	Let $F$ be a Heegaard surface in some $3$-manifold which is compressible to both sides.
	The surface $F$ is \textit{critical} if the set of all compressing disks for $F$ can be partitioned into subsets $C_0$ and $C_1$ such that the follows hold.
	\begin{enumerate}
		\item For each $i=0,1$ there is at least one weak reducing pair $(D_i,E_i)$, where $D_i$, $E_i\in C_i$.
		\item If $D\in C_i$ and $E\in C_j$, then $(D,E)$ is not a weak reducing pair for $i\neq j$.
	\end{enumerate}
\end{definition}

Note that the definition of \textit{``critical surface''} of \cite{Bachman2} is significantly simpler and slightly weaker, than the one given in \cite{Bachman1}.
In other words, anything that was considered critical in \cite{Bachman1} is considered critical here as well.

Let $M$ be a compact orientable $3$-manifold with possibly non-empty boundary, and suppose that $M$ has a weakly reducible Heegaard splitting $H=(V,W;F)$. 

Let $D$ and $D'$ be compressing disks in $V$.
By abuse of terminology, we will just say that $D\cap D'=\emptyset$ ($E\cap E'=\emptyset$ resp.) if they are not isotopic in $V$ ($W$ resp.) and $D$ misses $D'$ ($E$ misses $E'$ resp.).
In the case $D$ and $D'$ ($E$ and $E'$ resp.) are isotopic in $V$ ($W$ resp.), we will denote it as $D=D'$ ($E=E'$ resp.) even if $D$ misses $D'$ ($E$ misses $E'$ resp.)
Similarly, we will just say that $\partial D\cap \partial E=\emptyset$ if $\partial D$ and $\partial E$ are not isotopic in $F$ and $\partial D$ misses $\partial E$.
In the case $\partial D$ and $\partial E$ are isotopic in F, we will denote it as $\partial D=\partial E$ even if $\partial D$ misses $\partial E$.
Therefore, for a weak reducing pair $(D,E)$, we get either $\partial D\cap \partial E=\emptyset$ or $\partial D=\partial E$.
If there is a weak reducing pair which holds the latter case, the splitting is reducible. 

\begin{definition}
	Define \textit{the complex of weak reducing pairs $\WR$} as follows.
	\begin{enumerate}
		\item Each vertex of $\WR$ is a weak reducing pair for $F$.
		\item Two vertices $v=(D,E)$ and $w=(D',E')$ in $\WR$ are the same if and only if $D=D'$ and $E=E'$.
		\item Assign an edge (or a $1$-simplex) $e$ between two different vertices $v=(D,E)$ and $w=(D',E')\in \WR$ if (a) $D=D'$ and $E\cap E'=\emptyset$, or (b) $D\cap D'=\emptyset$, and $E=E'$.
		The edge $e$ between $v$ and $w$ is determined uniquely since it is impossible that (a) and (b) occur simultaneously. 
	\end{enumerate}
	We define an \textit{$n$-simplex} in $\WR$ by $(n+1)$-vertices $v_0$, $\cdots$, $v_n$ if there is an edge between $v_i$ and $v_j$ for every choice of $i$ and $j$, where $0\leq i\neq j\leq n$.
	$\WR$ depends on the manifold and the Heegaard splitting.
	If we need to make sure of the Heegaard surface or the manifold, denote the complex as $\WR(F)$ or $\WR(M;F)$.
\end{definition}

It may be possible that an edge $(D_1,E)-(D_2,E)$ exists but $\partial D_1 = \partial D_2$ if the manifold has a sphere boundary.
But if we consider irreducible manifolds other than $B^3$, then a Heegaard splitting consists of non-punctured compression bodies.
Therefore, if $M$ is an irreducible manifold other than $B^3$ and $\partial D_1 = \partial D_2$, then we can isotope $D_1$ and $D_2$ in $V$ so that $D_1 = D_2$  since a compression body is irreducible.
Similary, if $M$ is an irreducible manifold other than $B^3$, $\partial D_1 \cap \partial D_2=\emptyset$, and $\partial D_1$ is not isotopic to $\partial D_2$ in $F$, then we can isotope $D_1$ and $D_2$ in $V$ so that  $D_1\cap D_2=\emptyset$. 

In section 8 of \cite{Bachman2}, D. Bachman considered a sequence of compressing disks 
$$D=D_0 - E=E_0 - D_1 - E_1 - \cdots - D'=D_m - E'=E_m,$$
where (a) $D_i = D_{i+1}$ or $D_i\cap D_{i+1}=\emptyset$ and (b) $E_i = E_{i+1}$ or $E_i\cap E_{i+1}=\emptyset$ for each $0\leq i \leq m-1$ and both $(D_i, E_i)$ for $0\leq i\leq m$ and $(D_{i+1},E_i)$ for $0\leq i \leq m-1$ are weak reducing pairs.
He defined the distance between two weak reducing pairs $(D,E)$ and $(D',E')$ using the minimal length of this sequence.
In particular, if there is no such a   sequence between them, the distance is defined as $\infty$.
He also proved that if the distance between two weak reducing pairs is $\infty$, then the Heegaard surface is critical.
We can rewrite this result in terms of $\WR$ as follows.

\begin{proposition}[D. Bachman, Lemma 8.5 of \cite{Bachman2}]\label{prop-Bachman}
	If $\WR$ is disconnected, then the Heegaard surface is critical.
\end{proposition}

\begin{proof}
	We will prove that two vertices $(D,E)$ and $(D',E')$ in $\WR$ are connected by a union of $1$-simplices if and only if the distance between $(D,E)$ and $(D',E')$ is finite. Therefore, this is equivalent to  Lemma 8.5 of \cite{Bachman2}.

	Suppose that $(D,E)$ and $(D',E')$ in $\WR$ are connected by a union of $m$ $1$-simplices,
	$$(D=D_0,E=E_0)-_{e_1}(D_1,E_1)-_{e_2}\cdots-_{e_m}(D'=D_m,E'=E_m).$$
	Each $1$-simplex $(D_{i-1},E_{i-1})-_{e_i}(D_i,E_i)$ denotes that (a) $D_{i-1} = D_i$ and $E_{i-1}\cap E_i=\emptyset$ or (b) $D_{i-1} \cap D_i=\emptyset$ and $E_{i-1}= E_i$.
	In both cases, we can assign a sequence like following to $e_i$,
	$$D_{i-1}-E_{i-1}-D_{i-1}-E_i-D_i-E_i.$$
	Similarly, for the $1$-simplex $(D_{i},E_{i})-_{e_{i+1}}(D_{i+1},E_{i+1})$, we can assign a sequence like following to $e_{i+1}$,
	$$D_{i}-E_{i}-D_{i}-E_{i+1}-D_{i+1}-E_{i+1}.$$
	The sequence corresponding to $e_i$ ends with $D_i-E_i$ and that corresponding to $e_{i+1}$ starts with $D_i-E_i$.
	Therefore, if we connect these two sequences identifying the common segment $D_i - E_i$ for all $1\leq i\leq m-1$, then we get the wanted sequence, i.e. the distance between $(D,E)$ and $(D',E')$ is finite.

	Conversely, suppose that the distance between $(D,E)$ and $(D',E')$ is finite, i.e. there is a sequence like following,
$$D=D_0 - E=E_0 - D_1 - E_1 - \cdots - D'=D_m - E'=E_m.$$
	Now we will follow the procedure described below.
\begin{enumerate}
	\item Read the sequence from $D_0 - E_0$. Initially, we assign a vertex $(D_0,E_0)$.
	\item \label{bachman85-proof-2} For $D_0-E_0-D_1$, if $D_0\cap D_{1}=\emptyset$, then we assign a $1$-simplex $e$ in $\WR$, $(D_0,E_0)-_{e}(D_1,E_0)$.
	If $D_0=D_1$, then we define $(D_1,E_0)$ as $(D_0,E_0)$.
	\item \label{bachman85-proof-3} For $E_0-D_1-E_1$, if $E_0\cap E_{1}=\emptyset$, then we assign a $1$-simplex $f$ in $\WR$, $(D_1,E_0)-_{f}(D_1,E_1)$.
	If $E_0=E_1$, then we define $(D_1,E_1)$ as $(D_1,E_0)$.
	\item We repeat the steps  (\ref{bachman85-proof-2}) and (\ref{bachman85-proof-3}), i.e. consider $D_i-E_i-D_{i+1}$ and $E_i-D_{i+1}-E_{i+1}$ for all $1\leq i \leq m-1$.
	\item Finish after the step $E_{m-1}-D_m-E_m$.
\end{enumerate}
Here, we get a union of $1$-simplices which connects $(D,E)$ and $(D',E')$ in $\WR$.
\end{proof}

In terms of $\WR$, we can also get the following lemma.

\begin{lemma}[D. Bachman, Lemma 8.4 of \cite{Bachman2}]\label{lemma-Bachman}
	If there are two different vertices $v=(D,E)$ and $w=(D,E')$, or $v=(D,E)$ and $w=(D',E)$ in $\WR$, then there is a path from $v$ to $w$ in $\WR$.
\end{lemma}

We denote a 2-simplex $\Delta$ determined by the three vertices $u$, $v$, $w$ in $\WR$ as $u-_e v-_f w-_g u$, where $e$ is the $1$-simplex in $\Delta$ between $u$ and $v$, $f$ is that between $v$ and $w$, and $g$ is that between $w$ and $u$.

\begin{lemma}\label{lemma1}
	Suppose that $\Delta$ is a 2-simplex in $\WR$ determined by three vertices $u$, $v$, $w$.
	Then $\Delta$ has the form $(D,E)-_e(D,E')-_f(D,E'')-_g(D,E)$ or  $(D,E)-_e(D',E)-_f(D'',E)-_g(D,E)$.
\end{lemma}

\begin{proof}
	If $\Delta$ does not have the form $(D,E)-_e(D,E')-_f(D,E'')-_g(D,E)$ or  $(D,E)-_e(D',E)-_f(D'',E)-_g(D,E)$, then we get two cases by fixing the initial and terminal vertices as $(D,E)$ and reading $u-_e v-_f w-_g u$ from the left.

	\Case{1} $\Delta=(D,E)-_e(D,E')-_f(D',E')-_g(D,E)$.\\
	From the edge $e$, $E\cap E'=\emptyset$.
	Therefore, we get $D= D'$ from the edge $g$, this contradicts the existence of the edge $f$.

	\Case{2} $\Delta=(D,E)-_e(D',E)-_f(D',E')-_g(D,E)$.\\
	From the edge $e$, $D\cap D'=\emptyset$.
	Therefore, we get $E=E'$ from the edge $g$, this contradicts the existence of the edge $f$.
\end{proof}

\begin{lemma}\label{lemma2}
	Suppose that $\Sigma$ is an $n$-simplex in $\WR$ determined by $v_0$, $\cdots$, $v_n$.
	Then the vertices of $\Sigma$ have the form (a) $v_0=(D_0,E)$, $\cdots$ $v_n=(D_n,E)$ or (b) $v_0=(D,E_0)$, $\cdots$, $v_n=(D,E_n)$.
\end{lemma}

\begin{proof}
	We will use an induction argument for $n$.
	For $n=2$, we already proved in Lemma \ref{lemma1}.
	Suppose that the statement of Lemma \ref{lemma2} holds for $n=k\geq 2$, and consider a $(k+1)$-simplex $\Sigma$ determined by $v_0$, $\cdots$, $v_{k+1}$.
	Let $\Delta$ be the $k$-subsimplex of $\Sigma$ determined by $v_0$, $\cdots$, $v_k$, and $\Delta'$ be the $k$-subsimplex of $\Sigma$ determined by $v_0$, $\cdots$, $v_{k-1}$, $v_{k+1}$.
	If we consider $\Delta$, then we get the following two cases.
	\begin{enumerate}
		\item \label{lemma-case-1}the vertices of $\Delta$ have the form $v_0=(D_0, E)$, $\cdots$, $v_k=(D_k,E)$, or
		\item \label{lemma-case-2}the vertices of $\Delta$ have the form $v_0=(D, E_0)$, $\cdots$, $v_k=(D,E_k)$.
	\end{enumerate}
	In the case (\ref{lemma-case-1}), the vertices of $\Delta'$ have the form
	$$v_0=(D_0, E),\, \cdots,\,v_{k-1}=(D_{k-1},E),\, v_{k+1}.$$
	Since $k\geq 2$, we can assume the existence of $v_1=(D_1,E)$.
	By the induction hypothesis, $v_{k+1}$ must have the form $(D_{k+1},E)$.
	Similarly, in the case (\ref{lemma-case-2}), the vertices of $\Delta'$ have the form
	$$v_0=(D, E_0),\, \cdots,\,v_{k-1}=(D,E_{k-1}),\, v_{k+1}=(D,E_{k+1}).$$
	Therefore, the statement holds when $n=k+1$.
\end{proof}

Let $\bar{D}\subset V$, $\bar{E}\subset W$ be compressing disks, where $\partial \bar{D}$ and $\partial\bar{E}$ intersect transversely in a single point.
Such a pair of disks is called a \textit{canceling pair} of disks for the splitting.
If there is a canceling pair, then we call the splitting \textit{stabilized}.

Let $F$ be a surface and $D$ be a compressing disk of $F$.
If we compress $F$ along $D$, we call the compressed surface \textit{``$F_D$''}. Lemma \ref{lemma3} gives a connection of a vertex in $\WR$ with a stabilized Heegaard splitting.

\begin{lemma}\label{lemma3}
	If $\{\bar{D},\bar{E}\}$ is a canceling pair of a Heegaard splitting $H=(V,W;F)$ and the genus of $F$ is at least two.
	Then, there is a weak reducing pair $(D,E)$ such that $D=\bar{D}$ or $E=\bar{E}$.
\end{lemma}

\begin{proof}
	Let $C$ the boundary of a regular neighborhood of $\partial \bar{D}\cup\partial \bar{E}$ in $F$.
	Then there is a simple closed curve $C'$ in $F$ in the same isotopy class of $C$ in $F$, such that $C'$ bounds a disk in $F_{\bar{E}}$ ($F_{\bar{D}}$ resp.).
	If we push the interior of this disk slightly into $W$ ($V$ resp), we get a disk $E$ ($D$ resp), such that $\bar{D}\cap E=\emptyset$ ($\bar{E}\cap D=\emptyset$ resp).
	If we choose $D=\bar{D}$ ($E=\bar{E}$ resp), then we get the wanted weak reducing pair $(D,E)$ (see Figure \ref{fig-cancelling}.)
\end{proof}

\begin{figure}
	\includegraphics[width=6cm]{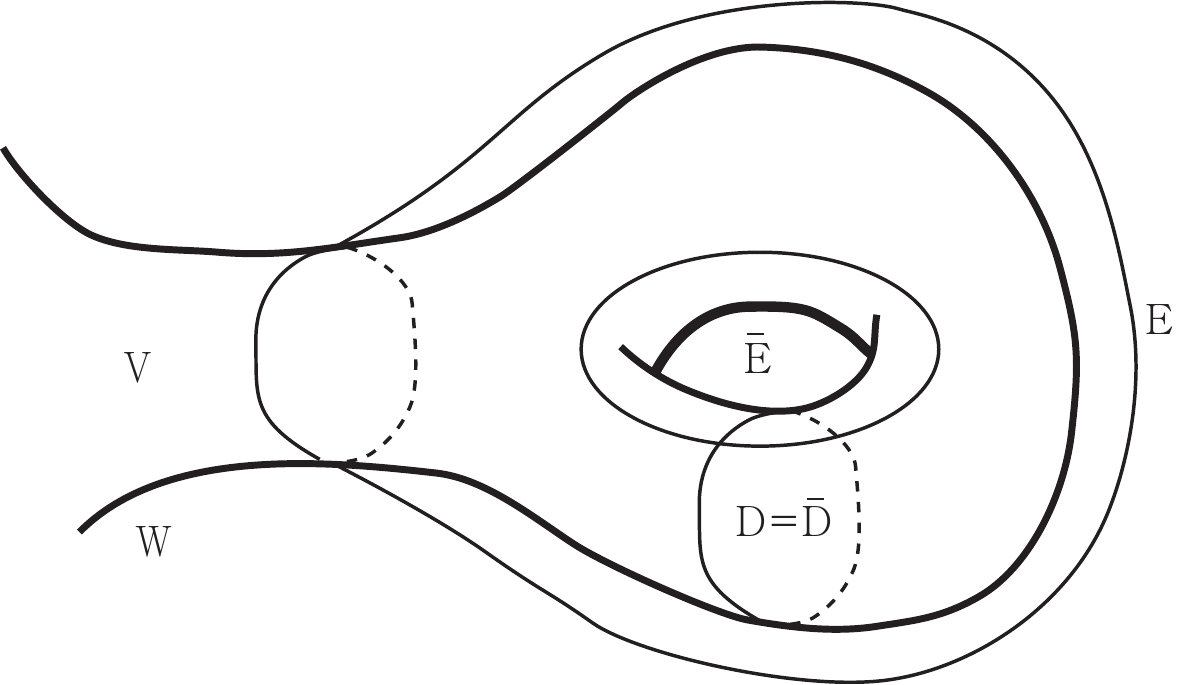}
	\caption{We can take a weak reducing pair $(D,E)$ such that  $D=\bar{D}$. \label{fig-cancelling}}
\end{figure}

\begin{definition}\label{definition4}
	Let $v=(D, E)$ be a vertex in $\WR$.
	\begin{enumerate}
		\item \label{def13-1} If there is a canceling pair $\{\bar{D}, \bar{E}\}$ such that $D=\bar{D}$, then we replace $F$ by $F_D$, and redefine $\WR(F)$ as $\WR(F_D)$.
		\item \label{def13-2} If there is a canceling pair $\{\bar{D}, \bar{E}\}$ such that $E=\bar{E}$, then we replace $F$ by $F_E$, and redefine $\WR(F)$ as $\WR(F_E)$.

	\end{enumerate}
	(Of course, the Heegaard splitting is also changed if one of (\ref{def13-1}) and (\ref{def13-2}) holds since the Heegaard surface is changed.)
	If one of (\ref{def13-1}) and (\ref{def13-2}) holds for a vertex $v$ in $\WR$, we will call this procedure a \textit{destabilization of $\WR$ for the vertex $v$}.
	We can repeat destabilizations of $\WR(F)$ until we cannot destabilize $\WR(F)$ for any vertex in $\WR(F)$.
	We call the resulting Heegaard splitting and $\WR(F)$ \textit{unstabilized}.
	If we need $n$-destabilization steps to make $\WR$ unstabilized, we call the initial Heegaard splitting \textit{$n$-stabilized} and this procedure an \textit{$n$-destabilization}.
\end{definition}

\begin{definition}\label{definition5}
	If there exists $n\in\mathbb{Z}$ such that $\dim(\Delta)\leq n$ for every simplex $\Delta$ in $\WR$ and $n$ is such a smallest integer, then we call $n$ \textit{the dimension of $\WR$}.
\end{definition}

If $\WR$ is $n$-stabilized, then it is easy to see that there are $n$ mutually disjoint non-isotopic essential disks $D_1$, $\cdots$, $D_n$ in $V$, and an essential disk $E$ in $W$ s.t. $\partial E\cap \partial D_i=\emptyset$ for $1\leq i\leq n$ (see Figure \ref{fig-nst}.) 
Therefore we get an $(n-1)$-simplex by $n$-vertices $(D_1, E)$, $\cdots$, $(D_n,E)$ in $\WR$, i.e. the dimension of $\WR$ is at least $n-1$.

\begin{figure}
	\includegraphics[width=6cm]{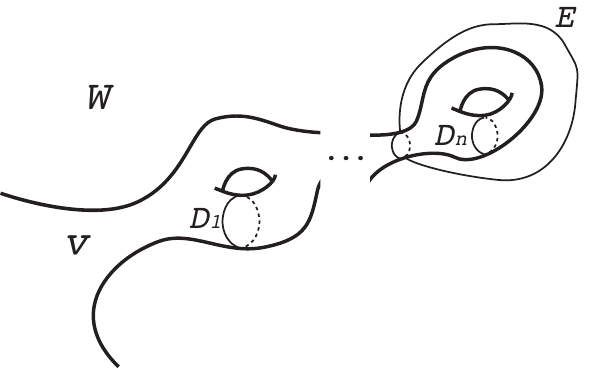}
	\caption{If $\WR$ is $n$-stabilized, then there is an $(n-1)$-simplex determined by $(D_1, E)$, $\cdots$, $(D_n, E)$.\label{fig-nst}}
\end{figure}

\section{The proof of Theorem \ref{theorem-main}\label{section-proof-main-theorem}}

In this section, we will discuss about the proof of Theorem \ref{theorem-main}.

Assume $M$, $H$, and $F$ as in Theorem \ref{theorem-main}.
Since the only irreducible $3$-manifold which has a boundary component isomorphic to a $2$-sphere is $B^3$ and a Heegaard splitting of genus three of $B^3$ is stabilized by using Waldhausen's Theorem (see \cite{Waldhausen}), we do not consider $B^3$ and assume that a boundary component of $M$ is not homeomorphic to a $2$-sphere.

Suppose that there is an $m$-simplex $\Delta$ in $\WR$.
If we compress $F$ along the disks of $\Delta$, we get a union of surfaces.
We will prove that if the splitting $H$ is unstabilzed, then there is no $2$-sphere component determined by some disks of $\Delta$ which come from different compression bodies.

\begin{lemma} \label{lemma-EDsphere} 
	Suppose that $M$ is irreducible $3$-manifold and $H=(V,W;F)$ is an unstabilized Heegaard splitting of $M$.
	Let $\Delta$ be an $m$-simplex determined by the vertices $(D_0,E)$, $\cdots$, $(D_m, E)$.
	Then, $F_{D_0\cdots D_m E}$ cannot have a $S^2$ component determined by $E$ and some $D_i$'s.
	Similarly, if $\Delta$ is determined by $(D,E_0)$, $\cdots$, $(D, E_m)$, then we get the same result for $F_{D E_0\cdots E_m}$.
\end{lemma}

\begin{proof}
	Suppose that there is a $S^2$ component $S$ in $F_{D_0\cdots D_m E}$ determined by $E$ and some $D_i$'s.
	We can assume that 
$$S=\bar{S}\cup (\cup \bar{D}_i)\cup(\cup \bar{E}_j),$$
	where $\bar{S}$ is a planar surface, $\{\bar{D}_i\}_i$ is a set of parallel copies of the $D_i$'s, and $\{\bar{E}_j\}_j$ is a set of parallel copies of $E$.
	Note that $\{\bar{E}_j\}_j$ consists of a disk or two disks.
	Let us consider a simple closed curve $\alpha$ in the interior of $\bar{S}$ which separates $\cup\bar{D}_i$ and $\cup\bar{E}_j$ in $S$.
	$\alpha$ bounds two disks $B_1$ and $B_2$ in $S$, where $B_1$ contains $\cup\bar{D}_i$ and $B_2$ contains $\cup\bar{E}_j$. 
	If we isotope the interior of $B_1$ and $B_2$ slightly, then we can assume that $B_1$ is properly embedded in $V$ and $B_2$ is properly embedded in $W$. 
	Moreover $B_1$ and $B_2$ are essential disks in their compression bodies otherwise each of the $D_i$'s is inessential in $V$ or $E$ is inessential in $W$.
	Therefore, the splitting is reducible, i.e. the manifold is reducible or the splitting is stabilized (see Corollary 3.4.1 of \cite{SaitoScharlemannSchultens}.)
	This contradicts the hypothesis of the lemma.
\end{proof}

\begin{lemma}\label{lemma7}
	Suppose that $M$ is an irreducible $3$-manifold. Let $H=(V,W;F)$ be an unstabilized genus three Heegaard splitting of $M$. Then $\dim(\WR)\leq 1$.
\end{lemma}

\begin{proof}
	Let $v=(D_0, E)$ be a vertex in $\WR$, and without loss of generality, suppose that $v$ is contained in an $m$-simplex $\Delta$ determined by $m+1$ vertices $(D_0,E)$, $\cdots$, $(D_m, E)$.
	Now $\partial D_0$, $\cdots$, $\partial D_m$, and $\partial E$ are mutually disjoint and non-isotopic essential curves in $F$.
	(If some $\partial D_i$ is isotopic to $\partial E$ in $F$, then the splitting is reducible. But this leads that $M$ is reducible or the splitting is stabilized.)

	\Case{1} $\partial E$ is non-separating in $F$.

	\Case{1-1} $F_{D_0\cdots D_m E}$ is connected.\\
	This means that $\partial D_0$, $\cdots$, $\partial D_m$, and $\partial E$ are non-separating in $F$.
	Now $F_{D_0 \cdots D_m E}$ is a surface whose genus is at most one.
	If $F_{D_0 \cdots D_m E}$ is a torus, then we get $m=0$.
	Otherwise, $F_{D_0 \cdots D_m E}$ is a $2$-sphere and $m=1$, but this contradicts Lemma \ref{lemma-EDsphere}.

	\Case{1-2} $F_{D_0\cdots D_m E}$ has two or more components.\\
	In this case, there are one or two components in $F_{D_0\cdots D_m E}$ such that each of these components is determined by $E$ and some $D_i$'s.
	Let $\bar{F}_{D_0\cdots D_m E}$ be one of such components and suppose that $E$ and $D_0$, $\cdots$, $D_k$, without loss of generality, determine $\bar{F}_{D_0\cdots D_m E}$.
	If there are two such components, then we choose the one which makes $k$ smallest as $\bar{F}_{D_0\cdots D_m E}$.
	Assume that $D_j$ is superfluous to determine $\bar{F}_{D_0\cdots D_m E}$ if $j>k$.
	We can assume that the genus of $\bar{F}_{D_0\cdots D_m E}$ is at least one by Lemma \ref{lemma-EDsphere}.

	\Case{1-2-1} If $k=0$, then either $\partial E \cup \partial D_0$ cuts $F$ into two twice-punctured tori where $\partial D_0$ is also non-separating in $F$ or $\partial D_0$ is separating in $F$.

	Assume that the former case. Suppose that $m\geq 1$ and let $T$ be the twice-punctured torus component which $\partial D_1$ in contained in.
	Then, $\partial D_1$ is an essential simple closed curve in $T$ which is not $\partial$-parallel.
	If $\partial D_1$ is separating in $T$, then $\partial D_0\cup\partial D_1\cup\partial E$ cuts off a pair of pants $P$ from $F$.
	Hence, $F_{D_0 D_1 E}$ has a $S^2$ component determined by $D_0$, $D_1$, and $E$.
	If we apply Lemma \ref{lemma-EDsphere} to the $1$-subsimplex of $\Delta$ determined by $(D_0, E)$ and $(D_1, E)$, then we get a contradiction.
	If $\partial D_1$ is non-separating in $T$, then $F_{D_0 D_1 E}$ has a $S^2$ component determined by $D_0$, $D_1$, and $E$.
	Also, we get a contradiction similarly. 
	Therefore, we get $m=0$.

	Now we assume the latter case.
	Let $F_{D_0}=F'_{D_0}\cup F''_{D_0}$, where $g(F'_{D_0})=2$ and $g(F''_{D_0})=1$.
	Suppose that $\partial E\subset F''_{D_0}$.
	Since $\partial E$ is not isotopic to $\partial D_0$ in $F$, $\partial E$ must be essential in $F''_{D_0}$, i.e. $F_{D_0 E}$ has a $S^2$-component determined by $D_0$ and $E$.
	If we apply Lemma \ref{lemma-EDsphere} to the $0$-subsimplex of $\Delta$ determined by $(D_0, E)$, then we get a contradiction.
	Therefore, we get $\partial E\subset F'_{D_0}$.
	Moreover, $k=0$ implies $\partial D_1$, $\cdots$, $\partial D_m \subset F''_{D_0}$.
	Since there is only one isotopy class for non-trivial disjoint simple closed curves in $F''_{D_0}$, we get $m\leq 1$.
	(We do not consider the curves isotopic to $\partial D_0$.)
	In particular, $\partial D_1$ must be non-separating in $F$ if it exists.

	\Case{1-2-2} If $k\geq 1$ and at least two of $\partial D_0$, $\cdots$, $\partial D_m$ are separating in $F$.
	Assume that $\partial D_\alpha$ and $\partial D_\beta$ for $0\leq \alpha\neq \beta\leq m$ are separating in $F$. Then, $F_{D_\alpha D_\beta}= F'_{D_\alpha D_\beta} \cup F''_{D_\alpha D_\beta} \cup F'''_{D_\alpha D_\beta}$, where the three components of $F_{D_\alpha D_\beta}$ are genus one surfaces.
	Let $F'_{D_\alpha D_\beta}$ ($F'''_{D_\alpha D_\beta}$ resp.) be the one determined by only $D_\alpha$ ($D_\beta$ resp.) and $F''_{D_\alpha D_\beta}$ be the one determined by both $D_\alpha$ and $D_\beta$.
	Suppose that $\partial E\subset F'_{D_\alpha D_\beta}$ or $F'''_{D_\alpha D_\beta}$.
	In this case, $\partial E$ is essential in $F'_{D_\alpha D_\beta}$ or $F'''_{D_\alpha D_\beta}$, respectively.
	This means that we get a $2$-sphere component in $F_{D_\alpha D_\beta E}$ determined by $E$ and $D_\alpha$, or $E$ and $D_\beta$.
	Therefore, if we apply Lemma \ref{lemma-EDsphere} to the $1$-subsimplex of $\Delta$ determined by $(D_\alpha, E)$ and $(D_\beta,E)$, then we get a contradiction.
	Now we suppose that $\partial E\subset F''_{D_\alpha D_\beta}$.
	We can assume that there are two disks $D_\alpha'$ and $D_\beta'$ in $F''_{D_\alpha D_\beta}$ such that $D_\alpha'$ ($D_\beta'$ resp.) is a parallel copy of $D_\alpha$ ($D_\beta$ resp.).
	If $\partial E$ is inessential in $F''_{D_\alpha D_\beta}$, then $\partial E$ bounds a disk in $F''_{D_\alpha D_\beta}$ containing $D_\alpha'$ and $D_\beta'$ in its interior.
	In this case, $\partial E$ is separating in $F$, which contradicts the assumption of Case 1.
	If $\partial E$ is essential in $F''_{D_\alpha D_\beta}$, then we get a contradiction similarly as when $\partial E\subset F'_{D_\alpha D_\beta}$ or $F'''_{D_\alpha D_\beta}$.

	\Case{1-2-3} If $k\geq1$ and exactly one of $\partial D_0$, $\cdots$, $\partial D_m$, say $\partial D_\alpha$ for some $\alpha$ ($0\leq \alpha \leq m$), is separating in $F$.
	Then, $F_{D_\alpha}=F'_{D_\alpha}\cup F''_{D_\alpha}$, $g(F'_{D_\alpha})=2$, and $g(F''_{D_\alpha})=1$.
	Here, we get $\partial E\subset F'_{D_\alpha}$ by Lemma \ref{lemma-EDsphere}.
	Therefore, there is some $\partial D_j\subset F'_{D_\alpha}$ for $0\leq j\neq \alpha\leq m$ by the assumption $k\geq 1$.
	If we compress $F'_{D_\alpha}$ along $E$, then $(F'_{D_\alpha})_E$ is a torus.
	We can assume that there are three disks $E'$, $E''$, and $D_\alpha'$ in $(F'_{D_\alpha})_E$ where $E'$ ($E''$, $D_\alpha'$ resp.) is a parallel copy of $E$ ($E$, $D_\alpha$ resp.)
	If $\partial D_j$ is inessential in $(F'_{D_\alpha})_E$, then $\partial D_j$ bounds a disk in $(F'_{D_\alpha})_E$ containing (1) $E'$ and $E''$, (2) $E'$ and $D_\alpha'$, (3) $E''$ and $D_\alpha'$, or (4) $E'$, $E''$ and $D_\alpha'$ in its interior.
	In any case, there is a $S^2$-component in $F_{D_\alpha D_j E}$ determined by $D_j$ and $E$ or $D_\alpha$, $D_j$ and $E$.
	Hence, if we apply Lemma \ref{lemma-EDsphere} to the $1$-subsimplex of $\Delta$ determined by $(D_\alpha, E)$ and $(D_j, E)$, then we get a contradiction.
	If $\partial D_j$ is essential in $(F'_{D_\alpha})_E$, then we get a $S^2$-component in $F_{D_\alpha D_j E}$. 
	Therefore, We get a contradiction similarly.
	
	\Case{1-2-4} If $k\geq 1$ and all of $\partial D_0$, $\cdots$, $\partial D_m$ are non-separating in $F$, then we can find two of them $\partial D_\alpha$ and $\partial D_\beta$ for $0\leq\alpha\neq\beta\leq m$ such that $\partial D_\alpha\cup \partial D_\beta$ separate $F$. 
	(Suppose not, i.e. $F_{D_\alpha D_\beta}$ is connected for any choice of $\alpha$ and $\beta$ for $0\leq\alpha\neq\beta\leq m$.
	If $\partial E$ is inessential in $F_{D_\alpha D_\beta}$, then there is a $S^2$-component in $F_{D_\alpha D_\beta E}$ determined by $E$ and one or both of $D_\alpha$ and $D_\beta$.
	Hence, if we apply Lemma \ref{lemma-EDsphere} to the $1$-subsimplex of $\Delta$ determined by $(D_\alpha, E)$ and $(D_\beta,E)$, then we get a contradiction.
	Hence $\partial E$ is essential in $F_{D_\alpha D_\beta}$.
	Since $F_{D_\alpha D_\beta}$ is a torus, $F_{D_\alpha D_\beta E}$ is a $2$-sphere.
	Therefore, we get a contradiction similarly.)
	Now we get $F_{D_\alpha D_\beta}=F'_{D_\alpha D_\beta}\cup F''_{D_\alpha D_\beta}$, where each of $F'_{D_\alpha D_\beta}$ and $F''_{D_\alpha D_\beta}$ is a torus.
	Without loss of generality, assume that $\partial E\subset F'_{D_\alpha D_\beta}$.
	If $\partial E$ is inessential in $F'_{D_\alpha D_\beta}$, then there is a $S^2$-component in $F_{D_\alpha D_\beta E}$ determined by $D_\alpha$, $D_\beta$, and $E$.
	Hence, if we apply Lemma \ref{lemma-EDsphere} to the $1$-subsimplex of $\Delta$ determined by $(D_\alpha, E)$ and $(D_\beta,E)$, then we get a contradiction.
	If $\partial E$ is essential in $F'_{D_\alpha D_\beta}$, then this also contradicts to Lemma \ref{lemma-EDsphere} similarly.

	\Case{2} $\partial E$ is separating in $F$.

	In this case, $F_{E}=F'_{E}\cup F''_{E}$, where $g(F'_{E})=2$ and $g(F''_{E})=1$. By Lemma \ref{lemma-EDsphere}, $\partial D_0$, $\cdots$, $\partial D_m\subset F'_{E}$.

	\Case{2-1} If some $\partial D_\alpha$ for $0\leq \alpha \leq m$ is separating in $F$, then $F_{D_\alpha E}=F'_{D_\alpha E}\cup F''_{D_\alpha E} \cup F'''_{D_\alpha E}$ where each component is a torus.
	Let $F'_{D_\alpha E}$ ($F'''_{D_\alpha E}$ resp.) be the component determined by only $D_\alpha$ ($E$ resp.) and  $F''_{D_\alpha E}$ be the component determined by both $D_\alpha$ and $E$.
	If $m=0$, then proof ends.
	 Now we assume that $m\geq 1$.
	Suppose that there is $\partial D_\beta$ for $0\leq \beta\neq \alpha\leq m$  which inessential in $F_{D_\alpha E}$.
	If $\partial D_\beta\subset F'_{D_\alpha E}$ or $F'''_{D_\alpha E}$, then $\partial D_\beta$ is isotopic to $\partial D_\alpha$ or $\partial E$ respectively, i.e. a contradiction.
	Therefore, $\partial D_\beta\subset F''_{D_\alpha E}$.
	In this case, there is a $S^2$-component in $F_{D_\alpha D_\beta E}$ determined by $D_\alpha$, $D_\beta$, and $E$.
	If we apply Lemma \ref{lemma-EDsphere} to the $1$-subsimplex of $\Delta$ determined by $(D_\alpha, E)$ and $(D_\beta,E)$, then we get a contradiction.
	Therefore, $\partial D_\beta$ is essential in $F_{D_\alpha E}$ for $0\leq \beta\neq \alpha\leq m$.
	Since all components of $F_{D_\alpha E}$ are tori, Lemma \ref{lemma-EDsphere} forces $\partial D_\beta$ to be in $F'_{D_\alpha E}$.
	Moreover, we get $m\leq 1$ since $\partial D_\beta$ is not isotopic to $\partial D_\alpha$ in $F$ and there is only one isotopy class of disjoint essential simple closed curves in $F'_{D_\alpha E}$.
	In particular, $\partial D_\beta$ is non-separating in $F$ if $m=1$.

	\Case{2-2} If all $\partial D_i$ for $0\leq i \leq m$ is non-separating in $F$, then $\partial D_0$, $\cdots$, $\partial D_m$ represent different isotopy classes of non-separating curves in $F'_{E}$.
	(Suppose that $\partial D_\alpha$ is isotopic to $\partial D_\beta$ in $F'_E$ for $0\leq\alpha\neq\beta\leq m$.
	Using this isotopy, we get an annulus which $\partial D_\alpha$ and $\partial D_\beta$ bound in $F'_{E}$, i.e. there is a $S^2$-component in $F_{D_\alpha D_\beta E}$ determined by $D_\alpha$, $D_\beta$, and $E$.
	Hence, if we apply Lemma \ref{lemma-EDsphere} to the $1$-subsimplex of $\Delta$ determined by $(D_\alpha, E)$ and $(D_\beta,E)$, then we get a contradiction.)
	If $m=0$, then the proof ends.
	If $m\geq 1$, then $(F'_{E})_{D_0 D_1}$ must contain a $S^2$-component determined by $D_0$, $D_1$, and $E$, i.e. $F_{D_0 D_1 E}$ does so.
	Hence, if we apply Lemma \ref{lemma-EDsphere} to the $1$-subsimplex of $\Delta$ determined by $(D_0, E)$ and $(D_1,E)$, then we get a contradiction.
\end{proof}

\begin{corollary}\label{corollary-pants}
	Assume $M$, $H$, and $F$ as in Lemma \ref{lemma7}.
	If $(D_0, E)$ and $(D_1, E)$ determine a $1$-simplex in $\WR$, then $\partial D_0$ is separating and $\partial D_1$ is non-separating in $F$ without loss of generality.
	Moreover, $\partial D_0 \cup \partial D_1$ cuts off a pair of pants from $F$, and $\tilde{F}$ does not have compressing disks in $V$ other than those parallel to $D_0$ where $\tilde{F}\subset F$ is obtained from the once-punctured genus two component of $F-\partial D_0$ removing $\partial E$. 
	We get the same result if $(D, E_0)$ and $(D, E_1)$ determine a $1$-simplex in $\WR$.
\end{corollary}

\begin{proof}
	In the proof of Lemma \ref{lemma7}, we have two cases when $\WR$ has a $1$-simplex, i.e. $m=1$, Case 1-2-1 and Case 2-1.
	In each case, one of $\partial D_0$ and $\partial D_1$ is separating and the other is non-separating in $F$.
	(In Case 1-2-1, the separating one is $\partial D_0$ and the non-separating one is $\partial D_1$.
	In Case 2-1, the separating one is $\partial D_\alpha$ and the non-separating one is $\partial D_\beta$.)
	Moreover, the separating one cuts off a once-punctured torus from $F$, and the non-separating one is a non-separating simple closed curve in this torus. 
	Therefore, $\partial D_0\cup \partial D_1$ cuts off a pair of pants from $F$ (see Figure \ref{fig-1-simplices}, the left one is when $\partial E$ is non-separating and the right one is when $\partial E$ is separating in $F$.)
	Since $\dim(\WR)\leq 1$, the last statement is obvious.
\end{proof}

\begin{figure}
	\includegraphics[width=6cm]{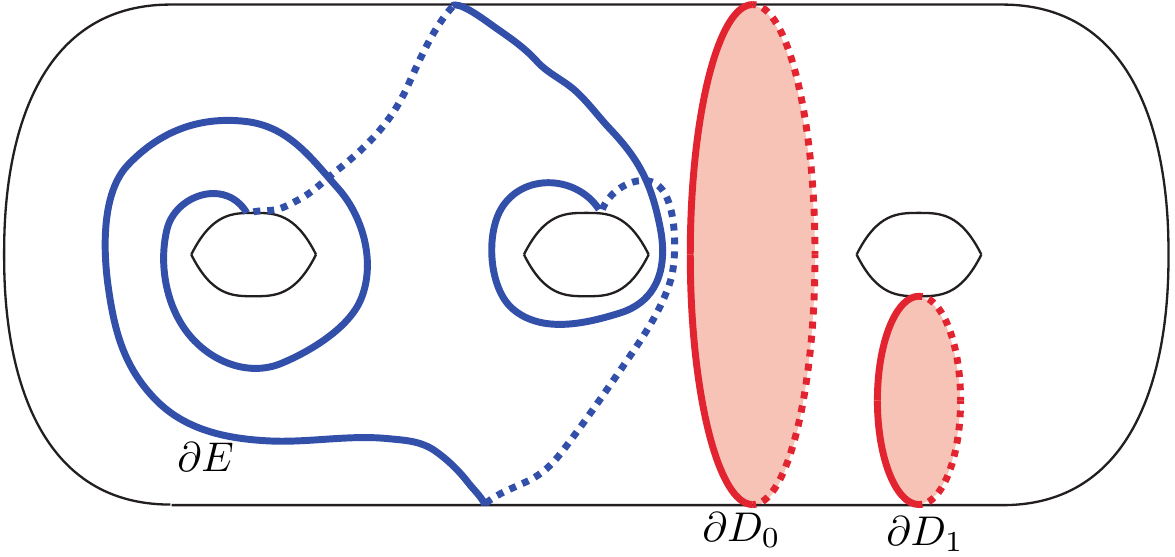}\quad\includegraphics[width=6cm]{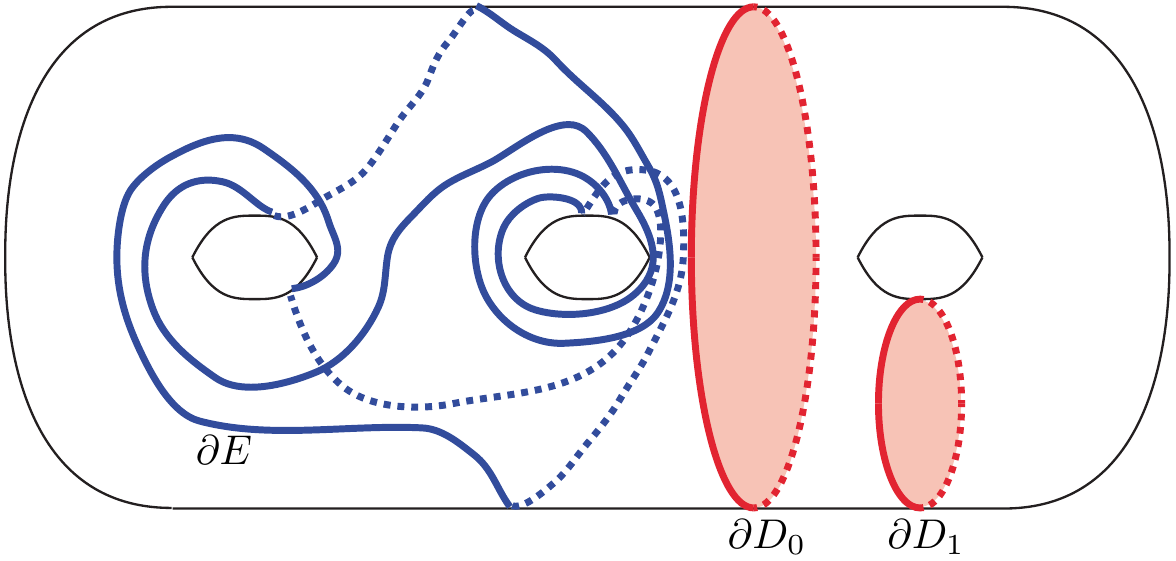}
	\caption{$1$-simplex cases \label{fig-1-simplices}}
\end{figure}

By Lemma \ref{lemma7} we can assume that $\WR$ is a $1$-dimensional graph if it is not a set of vertices.
Suppose that two vertices $v_1=(D_1,E_1)$ and $v_2=(D_2, E_2)$ determine an edge $e$ in $\WR$.
Let us label $e$ ``D'' (``E'' resp.) if $E_1=E_2$ ($D_1=D_2$ resp.) and call $e$ a \textit{D-edge} (an \textit{E-edge} resp.)

\begin{lemma}\label{lemma8}
	Assume $M$, $H$, and $F$ as in Lemma \ref{lemma7}.
	If there are two adjacent edges $e$ and $f$ in $\WR$ such that $e$ is determined by $(D_0, E)$ and $(D_1, E)$ and $f$ is determined by $(D_1, E)$ and $(D_2, E)$, i.e. both edges are labelled ``D'', then $\partial D_1$ is non-separating, and $\partial D_0$, $\partial D_2$ are separating in $F$.
	We get the same result if both edges are labelled E.
\end{lemma}

\begin{proof}
	By Corollary \ref{corollary-pants}, we have the following two cases.

	\Case{1} $\partial D_0$ and $\partial D_2$ are non-separating, and $\partial D_1$ is separating in $F$.

	\Case{2} $\partial D_0$ and $\partial D_2$ are separating, and $\partial D_1$ is non-separating in $F$.

	We will show that we only get Case 2.
	Suppose that Case 1 holds.
	Let $F_{D_1}=F'_{D_1}\cup F''_{D_1}$, where $g(F'_{D_1})=1$ and $g(F''_{D_1})=2$.
	Let $V'$ be the part of $V$ that $F'_{D_1}$ bounds in $V$.
	If we consider the edge $e$, then $\partial D_0\subset F'_{D_1}$ since $\partial D_0\cup \partial D_1$ cuts off a pair of pants from $F$ by Corollary \ref{corollary-pants}.
	The existence of the edge $e$ implies $D_0\cap D_1=\emptyset$, i.e. $D_0\subset V'$.
	Hence, we can compress $F'_{D_1}$ along the disk $D_0$ in $V'$ and get a $2$-sphere $S$.
	Since $\partial_-V$ does not have a $2$-sphere component, the inside of $S$ in $V$ must be a $3$-ball, i.e. $V'$ is a solid torus and $D_0$ is a meridian disk of $V'$.
	If we use Corollary \ref{corollary-pants} for the edge $f$, then $D_2$ is also a properly embedded disk in $V'$ and both of $D_0$ and $D_2$  must be meridian disks of $V'$.
	Therefore, $D_0$ is isotopic to $D_2$ in $V'$ and we can assume that this isotopy misses $D_1$.
	This means $D_0$ is isotopic to $D_2$ in $V$, which is a contradiction.
\end{proof}

\begin{definition}
	If there is a maximal connected subgraph of $\WR$ whose edges are labelled only D and it consists of two or more edges, we call it a \textit{``D-cluster''}.
	Similarly, we can define an \textit{``E-cluster''}.
	In a D-cluster (an E-cluster resp.), all vertices have the same $E$-disk ($D$-disk resp.). 
\end{definition}

\begin{lemma}\label{lemma-cluster}
	A D-cluster does not contain a loop.
	Moreover, there is only one vertex in a D-cluster which is adjacent to two or more edges in the D-cluster.
	We get the same result for an E-cluster.
\end{lemma}

\begin{proof}
	If a D-cluster have a subgraph like $v'-_{e_1}v-_{e_2}w-_{e_3}w'$ where $e_1$ and $e_2$ share a vertex $v$ and $e_2$ and $e_3$ share a vertex $w$ ($v\neq w$), then the boundary of the $D$-disk of $v$ must be non-separating in $F$ by applying Lemma \ref{lemma8} to $v'-_{e_1}v-_{e_2}w$.
	But if we apply Lemma \ref{lemma8} to $v-_{e_2}w-_{e_3}w'$ again, then we get the previous $D$-disk is separating in $F$, which is a contradiction.
	Therefore, a D-cluster cannot have a line segment of length $3$ or more.
	By the definition of $\WR$, there is no loop of length at most two, i.e a D-cluster cannot have a loop.

	Suppose that $\bar{v}$ and $\bar{w}$ are different vertices in a D-cluster such that each of both is adjacent to two or more edges in the D-cluster, i.e. there are different two edges $e_v^1$ and $e_v^2$ ($e_w^1$ and $e_w^2$ resp.) in the D-cluster containing $\bar{v}$ ($\bar{w}$ resp.).
	Since a D-cluster is connected, there is a path $l$ from $\bar{v}$ to $\bar{w}$ in the D-cluster.
	We can assume that $l$ does not contain both $e_v^1$ and $e_v^2$ since there is no loop in a D-cluster.
	Similarly, $l$ does not contain both $e_w^1$ and $e_w^2$.
	Without loss of generality, assume that $e_v^1$ and $e_w^1$ are not contained in $l$.
	If we consider $e_v^1\cup l \cup e_w^1$, then it is a line segment of length at least three in the D-cluster, i.e. a contradiction.
	Therefore, there is only one vertex in a D-cluster which is adjacent to two or more edges in the D-cluster.
\end{proof}

\begin{definition} 
	By Lemma \ref{lemma-cluster}, there is a unique vertex in a D-cluster (an E-cluster resp.) adjacent to two or more edges in the cluster.
	We call it \textit{``the center of a D-cluster (an E-cluster resp.)''}.
	We call the other vertices \textit{``hands of a D-cluster (an E-cluster resp.).''} See Figure \ref{fig-dcluster}.
	Note that if an edge $v-_{e}w$ is contained in a D-cluster (an E-cluster resp.), then one of $v$ and $w$ is the center of the cluster and the other is a hand.
\end{definition}

\begin{figure}
	\includegraphics[width=4.5cm]{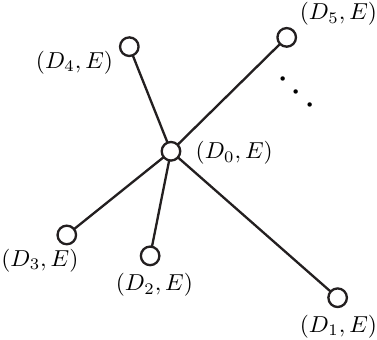}
	\caption{An example of a D-cluster in $\WR$. $(D_0, E)$ is the center and other vertices are hands. \label{fig-dcluster}}
\end{figure}

\begin{lemma}\label{lemma-DorE} 
	Assume $M$, $H$, and $F$ as in Lemma \ref{lemma7} and consider $\WR$.
	Every edge of $\WR$ is contained in some $D$-cluster or $E$-cluster. Moreover, every $D$- or $E$- cluster has infinitely many edges.
\end{lemma}

\begin{proof}
	Without loss of generality, assume that there is an edge $e$ labelled D, say $(D_0, E)-_{e}(D_1, E)$.
	By Corollary \ref{corollary-pants}, we can assume that $\partial D_0$ is separating and $\partial D_1$ is non-separating in $F$.
	If we compress $F$ along $D_1$, then $F_{D_1}$ bounds the part of $V$ which misses $D_1$, say $V'$, in $V$.
	Let $D_1'$ and $D_1''$ be two parallel copies of $D_1$ on $F_{D_1}$.
	Let us consider two disks $\bar{D}_1'$ and $\bar{D}_1''$ such that $\bar{D}_1'$ ($\bar{D}_1''$ resp.) is obtained by extending $D_1'$ ($D_1''$ resp.) slightly on $F_{D_1}$ and pushing its interior into $V'$ so that it is properly embedded in $V'$ (see Figure \ref{fig-arc}), and an arc $\alpha$ connecting $\partial\bar{D}_1'$ and $\partial\bar{D}_1''$ on $F_{D_1}$ which misses $\partial\bar{D}_1'$, $\partial\bar{D}_1''$, and $\partial E$ on its interior.
	If we perform the band sum of $\bar{D}_1'$ and $\bar{D}_1''$ along $\alpha$ in $V'$, then we get a properly embedded disk $D_0^\alpha$ in $V'$ which is $\partial$-parallel in $V'$.
	$D_0^\alpha$ is an essential separating disk in $V$ and we can assume that $D_0^\alpha$ misses $D_1$ and $E$.
	Therefore, we get an edge $(D_0^\alpha, E)-_{e_\alpha}(D_1, E)$ in $\WR$.
	We claim that there are infinitely many isotopy classes of $D_0^\alpha$ in $V$.
	If we compress $F_{D_1}$ again along $E$, then $\alpha$ is represented by an arc connecting $\partial\bar{D}_1'$ and $\partial\bar{D}_1''$ on $F_{D_1 E}$ which misses $\partial\bar{D}_1'$, $\partial\bar{D}_1''$, and the boundary of a parallel copy of $E$ (if $\partial E$ is separating in $F_{D_1}$) or two parallel copies of $E$ (if $\partial E$ is non-separating in $F_{D_1}$)  on its interior. 
	Note that the genus of each component of $F_{D_1 E}$ must be one, otherwise we can find a $S^2$-component in $F_{D_1 E}$ determined by $D_1$ and $E$ and we get a contradiction by applying Lemma \ref{lemma-EDsphere} to the $0$-simplex $(D_1, E)$.
	Let $\bar{F}_{D_1 E}$ be the component of $F_{D_1 E}$ containing $\alpha$. Since $g(\bar{F}_{D_1 E})=1$, there are infinitely many isotopy classes of such arcs (see Figure \ref{fig-arc3}, we do not fix the endpoints of the arcs on $\partial\bar{D}_1'$ and $\partial\bar{D}_1''$ when we consider these isotopies.)
	Moreover, if such arcs are non-isotopic, then the resulting disks are also non-isotopic in $V$.
\end{proof}

\begin{figure}
	\includegraphics[width=6.5cm]{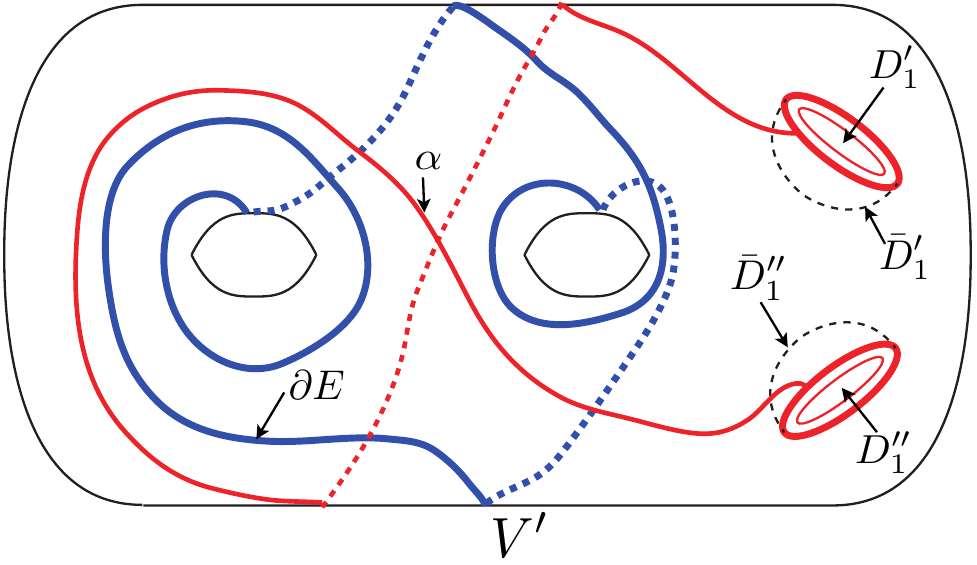}
	\caption{An arc $\alpha$ connecting $\partial\bar{D}_1'$ and $\partial\bar{D}_1''$\label{fig-arc}}
\end{figure}

\begin{figure}
	\includegraphics[width=6.5cm]{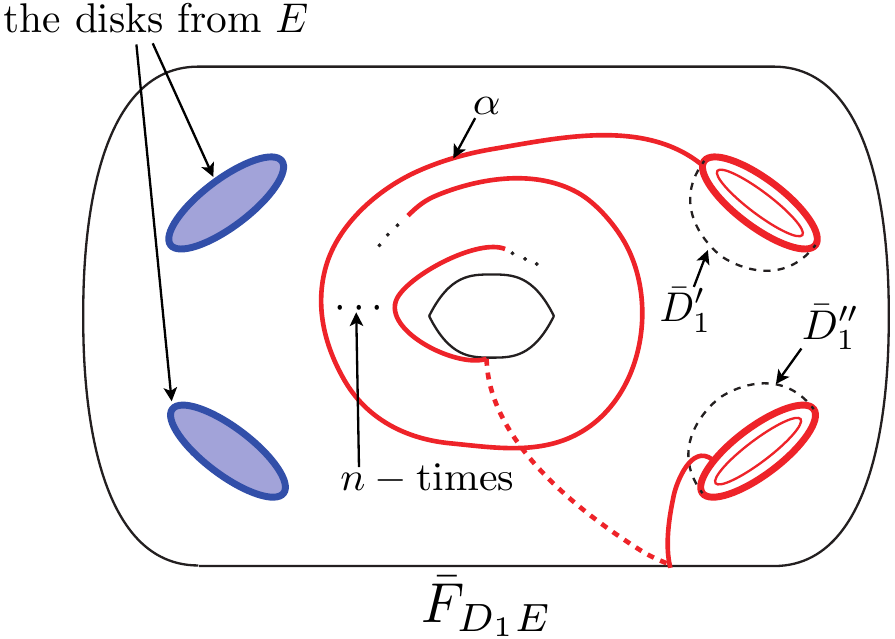}
	\caption{There are infinitely many isotopy classes of $\alpha$. \label{fig-arc3}}
\end{figure}

\begin{lemma}\label{lemma-difffourcurves}
	Assume $M$, $H$, and $F$ as in Lemma \ref{lemma7} and consider $\WR$.
	If there are consecutive two edges with different labels in $\WR$, then either the boundary curves of corresponding four disks are all disjoint or the D-cluster and the E-cluster containing the D-edge and the E-edge respectively have the common center.
	In the latter case, we can replace a separating disk among the four disks by a new separating disk in the same compression body so that the resulting four disks are all disjoint. 
\end{lemma}

\begin{proof}
	Let $(D_0, E_0)-_{e}(D_1, E_0)-_{f}(D_1, E_1)$ be the two edges, where $e$ is labelled D and $f$ is labelled $E$. By Corollary \ref{corollary-pants}, we get a situation like Figure  \ref{fig-1-simplices} for each edge, but in the case of $f$, we need to consider one D-disk and two E-disks.
	Since $D_0$, $D_1$, and $E_0$ are all disjoint, and $D_1$, $E_0$, and $E_1$ are also all disjoint by the existence of the edges $e$ and $f$, we will prove that $\partial E_1\cap \partial D_0=\emptyset$ or the both clusters guaranteed by Lemma \ref{lemma-DorE} have the common center.
	Let the separating one be $D_i$ and the non-separating one be $D_j$ between $D_0$ and $D_1$.
	Let $F_{D_i}=F'_{D_i}\cup F''_{D_i}$, where $g(F'_{D_i})=2$ and $g(F''_{D_i})=1$.
	We already know that $\partial E_0\subset F'_{D_i}$ and $\partial D_j\subset F''_{D_i}$ by Corollary \ref{corollary-pants}.

\Case{1} $\partial E_1\cap \partial D_j\neq \emptyset$, i.e. $j=0$. (so $i=1$)\\
	In this case, the existence of the vertex $(D_1, E_1)$ implies that $\partial E_1\cap \partial D_i=\emptyset$.
	Moreover, $\partial E_1\subset F''_{D_i}$ from $\partial E_1\cap \partial D_j\neq \emptyset$.
	But if $\partial E_1$ is non-separating in $F$, then $\partial E_1$ is essential in $F''_{D_i}$, i.e. $E_1$ and $D_i$ determine a $S^2$ component of $F_{D_i E_1}$. 
	Hence, if we apply Lemma \ref{lemma-EDsphere} to the $0$-simplex $(D_i, E_1)$, we get a contradiction.
	If $\partial E_1$ is separating in $F$, then it must be isotopic to $\partial D_i$, it is also a contradiction.

\Case{2} $\partial E_1\cap \partial D_j= \emptyset$.\\
	If $j=0$, then the proof ends. Therefore, assume that $j=1$ (so $i=0$.)
	If we apply Corollary \ref{corollary-pants} to the edge $f$, then one of $\partial E_0$ and $\partial E_1$ is separating and the other is non-separating in $F$.
	Let $\partial E_{k}$ be the separating one and $\partial E_{l}$ be the non-separating one.
	Here, $\partial D_1$ ($\partial E_{l}$ resp.) goes into the genus two component (the genus one component resp.) of $F_{E_k}$ again by Corollary \ref{corollary-pants}.

\Case{2-1} $\partial E_1 \cap \partial D_i \neq \emptyset$.\\
	Suppose that $l=1$, (so $k=0$). Let $F_{E_{k}}=F'_{E_{k}}\cup F''_{E_{k}}$, where $g(F'_{E_{k}})=2$ and $g(F''_{E_{k}})=1$.
	Here, $\partial D_i$ must go into $F'_{E_{k}}$ (if $\partial D_i \subset F''_{E_{k}}$, then it must be parallel to $\partial E_{k}$ since $\partial D_i$ is separating in $F$ and $g(F''_{E_{k}})=1$.)
	Since $\partial E_l\cap \partial D_i\neq \emptyset$ from the assumption of Case 2-1, $\partial E_l$ also goes into $F'_{E_{k}}$.
	But we already know that $\partial D_j=\partial D_1$ is contained in $F'_{E_{k}}$ from the assumption of Case 2.
	If we compress $F'_{E_{k}}$ along $D_j$ and $E_l$, then $(F'_{E_k})_ {D_{j} E_l}$ must have a $2$-sphere component determined by $D_j$, $E_{k}$, and $E_l$ since $\partial D_j$ and $\partial E_l$ are two disjoint non-separating simple closed curves in $F'_{E_k}$.
	Hence, if we apply Lemma \ref{lemma-EDsphere} to the $1$-simpex $f$, then we get a contradiction.
	Therefore we get $l=0$ (so $k=1$), i.e. both $\partial D_1$ and $\partial E_0$ ($\partial D_0$ and $\partial E_1$ resp.) are non-separating (separating resp.) in $F$.
	Let $\mathcal{D}$ ($\mathcal{E}$ resp.) be the D-cluster (the E-cluster resp.) containing the edge $e$ ($f$ resp.)
	Since the D-disk (E-disk resp.) for the center of a D-cluster (an E-cluster resp.) is non-separating in its compression body by Lemma \ref{lemma8}, the vertex $(D_0, E_0)$ ($(D_1, E_1)$ resp.) cannot be the center of $\mathcal{D}$ ($\mathcal{E}$ resp.)
	Therefore, $(D_1, E_0)$ is the common center of both $\mathcal{D}$ and $\mathcal{E}$.

	Now we will prove that there is another edge $(D_1,E_0)-_{g}(D_1,E_2)$ in $\mathcal{E}$ such that $E_2\cap D_0=\emptyset$.
	If we compress $F$ along $D_0$, then $F_{D_0}=F'_{D_0}\cup F''_{D_0}$ where $g(F'_{D_0})=2$ and $g(F''_{D_0})=1$.
	We get $\partial E_0\subset F'_{D_0}$ and $\partial D_1\subset F''_{D_0}$ by applying Corollary \ref{corollary-pants} to the edge $e$.
	If we compress $F'_{D_0}$ again along $E_0$, then we get a torus $\bar{F}_{D_0 E_0}$.
	There are two parallel copies $E_0'$ and $E_0''$ of $E_0$ and one parallel copy $D_0'$ of $D_0$ in $\bar{F}_{D_0 E_0}$.
	If we use the band sum arguments by using an arc connecting $E_0'$ and $E_0''$ missing $\partial E_0'$, $\partial E_0''$, and $\partial D_0'$ on its interior as in the proof of \ref{lemma-DorE}, then we can find an essential separating disk $E_2$ in $W$ such that $E_2\cap D_0=E_2\cap D_1=E_2\cap E_0=\emptyset$.
	Therefore, we get an edge $(D_1, E_0)-_g (D_1, E_2)$ which is contained in $\mathcal{E}$.
	By the existence of the edge $g$, three disks $D_1$, $E_0$, and $E_2$ are all disjoint.
	Since we already know that $D_0$, $D_1$, and $E_0$ are all disjoint by the existence of the edge $e$, $E_2\cap D_0=\emptyset$ means that $D_0$, $D_1$, $E_0$, and $E_2$ are all disjoint.

	\Case{2-2} $\partial E_1 \cap \partial D_i = \emptyset$.\\
	In this case, $\partial E_1 \cap \partial D_0 = \emptyset$, i.e. the proof ends.
\end{proof}

\begin{corollary}\label{corollary-difffourcurves}
	Assume $M$, $H$, and $F$ as in Lemma \ref{lemma7}. 
	There are consecutive two edges with different labels, $(\bar{D}, \tilde{E})-_{e}(\tilde{D}, \tilde{E})-_{f}(\tilde{D}, \bar{E})$ in $\WR$ if and only if the disks $\bar{D}$, $\tilde{D}\subset V$ and $\bar{E}$, $\tilde{E}\subset W$ hold the following conditions.
 		\begin{enumerate}
			\item \label{corollary-difffourcurves-1} Four boundary curves of the disks represent different isotopy classes in $F$.
			\item \label{corollary-difffourcurves-2}One of $\partial \bar{D}$ and $\partial \tilde{D}$ ($\partial \bar{E}$ and $\partial \tilde{E}$ resp.) is separating and the other is non-separating in $F$.	
			\item \label{corollary-difffourcurves-4} $\partial\bar{D}\cup \partial \tilde{D}$ cuts off a pair of pants from $F$, and so does $\partial\bar{E}\cup \partial \tilde{E}$. 
			Moreover, both pairs of pants are disjoint in $F$ if the four disks are all disjoint.
			\item \label{corollary-difffourcurves-3} Either (A) the four disks are all disjoint, or (B) we can replace one separating disk among these disks by another separating disk in the same compression body so that the resulting four disks are all disjoint and satisfying the conditions (\ref{corollary-difffourcurves-1}), (\ref{corollary-difffourcurves-2}), and (\ref{corollary-difffourcurves-4}).
		\end{enumerate}
\end{corollary}

\begin{proof}
	($\Leftarrow$) This is obvious since we can assume the four disks $\bar{D}$, $\tilde{D}$, $\bar{E}$, and $\tilde{E}$ are all disjoint.

	($\Rightarrow$) Suppose that there are consecutive two edges with different labels, $(\bar{D}, \tilde{E})-_{e}(\tilde{D}, \tilde{E})-_{f}(\tilde{D}, \bar{E})$ in $\WR$.
	If a D-disk and an E-disk share their boundaries, then the splitting is reducible, i.e. the manifold is reducible or the splitting is stabilized, this is a contradiction.
	Moreover, $\bar{D}$ and $\tilde{D}$ ($\bar{E}$ and $\tilde{E}$ resp.) cannot share their boundaries by the existence of edge $e$ ($f$ resp.).
	Therefore,  we get the condition (\ref{corollary-difffourcurves-1}).

	If we use Corollary \ref{corollary-pants} for both edges, we get the condition (\ref{corollary-difffourcurves-2}).

	Suppose that the four disks are all disjoint.
	Let $P_D$ ($P_E$ resp.) be the pair of pants which $\bar{D}\cup\tilde{D}$ ($\bar{E}\cup\tilde{E}$ resp.) cuts off from $F$ by Corollary \ref{corollary-pants}.
	If $P_D\cap P_E\neq \emptyset$, then one of both belongs to the interior of the other since the boundary curves are disjoint.
	Assume that $P_E\subset \operatorname{int}(P_D)$ without loss of generality.
	Then, either some component of $\partial P_E$ is inessential in $P_D$ (so also in $F$) or each component of $\partial P_E$ is isotopic to some component of $\partial P_D$ in $F$, any of both cases gives a contradiction.
	Therefore, we get the condition (\ref{corollary-difffourcurves-4}).
	
	If the four disks are all disjoint, then we get the condition (\ref{corollary-difffourcurves-3}).
	Assume that some disk among them intersects the union of the others.
	This means that Case 2-1 of the proof of  Lemma \ref{lemma-difffourcurves} holds.
	Therefore, we can assume that $\bar{D}$, $\tilde{D}$, and $\tilde{E}$ are all disjoint and $\tilde{D}$, $\tilde{E}$, and $\bar{E}$ are all disjoint, but $\bar{E}$ intersects $\bar{D}$.
	Moreover, we can replace $\bar{E}$ by another separating disk $\bar{E}'$ in $W$ so that we can make an edge $(\tilde{D},\tilde{E})-_g(\tilde{D},\bar{E}')$ and $\bar{D}$, $\tilde{D}$, $\tilde{E}$, and $\bar{E}'$ are all disjoint as in the proof.
	We can check that this four disks satisfy the conditions (\ref{corollary-difffourcurves-1}), (\ref{corollary-difffourcurves-2}), and (\ref{corollary-difffourcurves-4}) since the edges $e$ and $g$ are also consecutive two edges with different labels.
\end{proof}

We can imagine something like Figure \ref{fig-1-3} if there exist four disks satisfying the four conditions of Corollary \ref{corollary-difffourcurves} and they are all disjoint. The thick curves are separating and the thin curves are non-separating in $F$.

\begin{figure}
	\includegraphics[width=8cm]{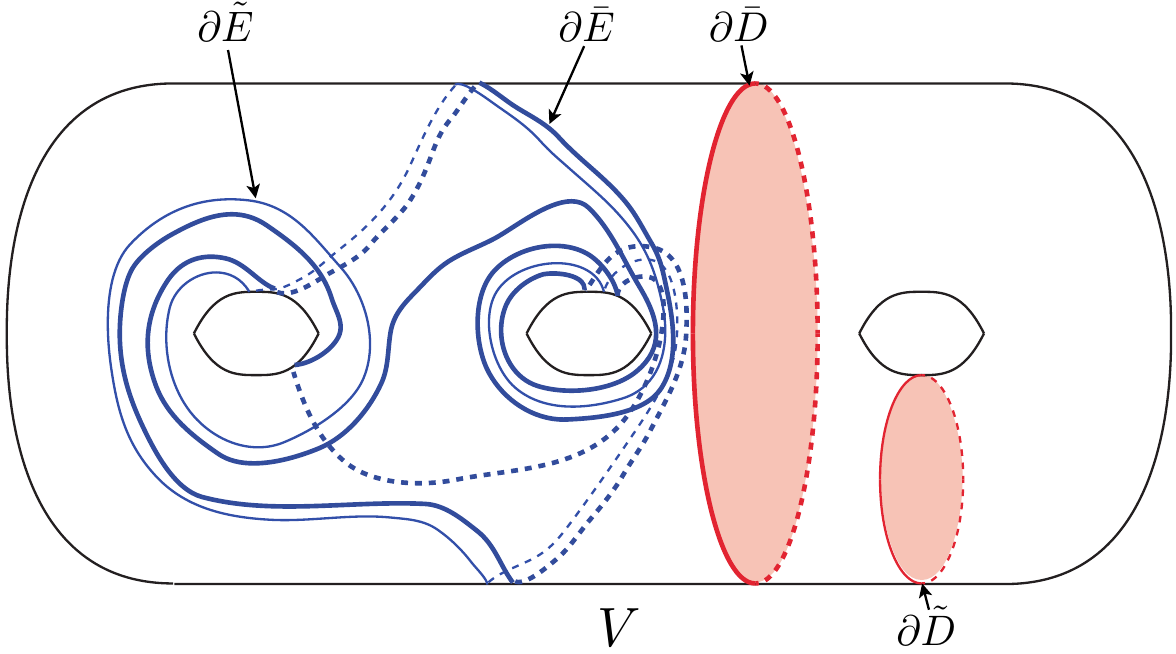}
	\caption{An example of the boundaries of four disks satisfying the four conditions of Corollary \ref{corollary-difffourcurves}\label{fig-1-3}}
\end{figure}

Now we introduce the key theorem of this article. 
The idea of this theorem is to isolate one cluster from another cluster in $\WR$. 

\begin{theorem}\label{theorem-main-1}
	Let $M$ be an irreducible $3$-manifold and $H=(V,W;F)$ be an unstabilized genus three Heegaard splitting of $M$.
	Suppose that there is no choice of four disks $\bar{D}$, $\tilde{D}\subset V$ and $\bar{E}$, $\tilde{E}\subset W$ which satisfy the four conditions of Corollary \ref{corollary-difffourcurves}.
	If we can choose two weak reducing pairs $(D_0, E_0)$ and $(D_1, E_1)$ such that $\partial D_0 \cap \partial E_1 \neq \emptyset$ and $\partial D_1 \cap \partial E_0\neq \emptyset$ up to isotopy, then $H$ is critical.
\end{theorem}

\begin{proof}
	If $(D_0, E_0)$ and $(D_1, E_1)$ represent the same vertex in $\WR$, then we get $D_0=D_1$.
	But $D_0\cap E_0=\emptyset$ leads $D_1\cap E_0=\emptyset$, which contradicts the assumption of this theorem.
	Hence, $(D_0, E_0)$ and $(D_1, E_1)$ are different vertices in $\WR$.
	If $\dim(\WR)=0$, then the proof ends by Proposition \ref{prop-Bachman}.
	Therefore, assume that $\dim(\WR)=1$ by Lemma \ref{lemma7}.
	We will prove that there is no path connecting $(D_0, E_0)$ and $(D_1, E_1)$ in $\WR$ to use Proposition \ref{prop-Bachman}.
	If there is a path from $(D_0, E_0)$ to $(D_1, E_1)$ in $\WR$, then each of both vertices belongs to some edge.
	Hence, Lemma \ref{lemma-DorE} forces each of both vertices to be contained in some D- or E- cluster.
	Now we get the three cases.

	\Case{1} $(D_0, E_0)$ and $(D_1, E_1)$ belong to the same D-cluster (E-cluster resp.).\\
	In this case, $E_0 = E_1$ ($D_0=D_1$ resp.).
	But the weak reducing pair $(D_0, E_0)$ means that $\partial D_0\cap \partial E_0=\emptyset$, i.e. $\partial D_0\cap \partial E_1=\emptyset$ ($\partial D_1\cap \partial E_0=\emptyset$ resp.), which contradicts the assumption of this theorem.

	\Case{2} $(D_0, E_0)$ belongs to some D-cluster (E-cluster resp.) and $(D_1, E_1)$ belongs to another D-cluster (E-cluster resp.)\\
	Suppose that there is a path connecting two vertices, $(D_0, E_0)-_{e_1}-\cdots -_{e_n}(D_1, E_1)$ in $\WR$.
	This path must contain an edge labelled E (D resp.) otherwise two clusters are united, but this goes to Case 1 and leads a contradiction.
	Therefore, we can find a line segment of length two in $\WR$ whose edges are labelled differently in the union of the path and these two clusters.
	This line segment gives the four disks satisfying the four conditions of Corollary \ref{corollary-difffourcurves}, which contradicts the assumption of this theorem.

	\Case{3} $(D_0, E_0)$ belongs to  some D-cluster (E-cluster resp.) and $(D_1, E_1)$ belongs to some E-cluster (D-cluster resp.).\\
	Suppose that there is a path connecting two vertices.
	Since two clusters have different labels, we can find a line segment of length two in $\WR$ whose edges are labelled differently in the union of the path and these two clusters, this gives a contradiction as in Case 2. 
\end{proof}

Now we will prove Theorem \ref{theorem-main}.
Let $M$, $H$, $F$, $D_0$, $D_1$, $E_0$, and $E_1$ be those in Theorem \ref{theorem-main}. 

We divide the proof into the two cases.

\Case{1} There is no weak reducing pair whose two disks are separating in their compression bodies.\\
If there is a choice of four disks satisfying the four conditions of Corollary \ref{corollary-difffourcurves}, 
then we can assume that these four disks are all disjoint, i.e. the two separating disks among them correspond to a weak reducing pair, which contradicts the assumption of Case 1.
Therefore, there is no choice of four disks satisfying the four conditions of Corollary \ref{corollary-difffourcurves}.
Hence, we can use Theorem \ref{theorem-main-1}.
This completes the proof.

\Case{2} There is a weak reducing pair whose two disks are separating in their compression bodies.

We will prove that we cannot choose four disks $\bar{D}$, $\tilde{D}\subset V$, $\bar{E}$, $\tilde{E}\subset W$ satisfying the four conditions of Corollary \ref{corollary-difffourcurves}.

Suppose that such four disks exist and assume that $\bar{D}\subset V$ and $\bar{E}\subset W$ are separating disks among these four disks.
(So $\tilde{D}$ and $\tilde{E}$ are non-separating in their compression bodies.) 

Since we can assume $\bar{D}\cap\bar{E}=\emptyset$, $\partial \bar{D}\cup \partial \bar{E}$ cuts $F$ into three pieces $F'$, $F''$, and $F'''$, where $F'$ ($F'''$ resp.) is the once-punctured torus determined by only $\partial \bar{D}$ ($\partial \bar{E}$ resp.) in $F$ and $F''$ is the twice-punctured torus determined by both $\partial \bar{D}$ and $\partial \bar{E}$ in $F$.
Since the Heegaard splitting $H$ must be irreducible (otherwise, $M$ is reducible or the splitting is stabilized), we can ignore the possibility that $\partial \bar{D}$ is isotopic to $\partial \bar{E}$ in $F$.
$\bar{D}$ ($\bar{E}$ resp.) cuts off a solid torus or $T^2\times I$ from $V$ ($W$ resp.) where $T^2$ is a torus.
By the assumption of this theorem, we can assume that at least one of $\bar{D}$ and $\bar{E}$, say $\bar{D}$ without loss of generality, cuts off $T^2\times I$ from $V$, say $V'$.
Let $T$ be the torus component of $\partial_-V$ which realizes the product $V'$.
Our assumption means that $V'$ is the part of $V$ which $F'\cup \bar{D}$ bounds in $V$.

Since $\partial \bar{D} \cup \partial \tilde{D}$ cuts off a pair of pants from $F$ from Corollary \ref{corollary-difffourcurves}, $\partial \tilde{D}$ must be a non-separating essential simple closed curve in the interior of $F'$, i.e. $\tilde{D}$ is a compressing disk for $F'=\partial V'\cap \partial_+V$.
Since $\bar{D}$ is a separating disk in $V$ and $\tilde{D}$ is disjoint from $\bar{D}$, we get $\tilde{D}\subset V'$, i.e. $F'$ is compressible in $V'\cong T^2\times I$, this is a contradiction 

Therefore, we can use Theorem \ref{theorem-main-1}. This completes the proof.

\section{The proof of Corollary \ref{corollary-main}\label{section-proof-corollary-main}}
In this section, we will prove Corollary \ref{corollary-main}. First, we consider the following lemmas.

\begin{lemma}\label{lemma-corollary}
	Let $M$ be a closed $3$-manifold, $H$ be a genus three Heegaard splitting of $M$, $F$ be the Heegaard surface. 
	$M$ has a weak reducing pair such that the two disks are separating in their handlebodies and the boundaries of both disks are not isotopic in $F$ if and only if $M=M_1\cup_T M_2$, where each $M_i$ has a genus two Heegaard splitting for $i=1,2$, $T$ is a torus, and $H$ can be represented as an amalgamation of these genus two splittings along $T$.
\end{lemma}

\begin{proof}
	$(\Leftarrow)$
	Let the Heegaard splitting of $M_1$ ($M_2$ resp.) be $(V_+, W_+;F_1)$ ($(V_-, W_-;F_2)$ resp.), and $W_+\cap V_-=T$.
	Since a compression body with non-empty minus boundary components can be represented as a union of a product of its minus boundary components and $1$-handles, $W_+ = (T\times I) \cup \text{(a $1$-handle)}$ and $V_- = (T\times I) \cup \text{(a $1$-handle)}$.
	Let $T$ be the $0$-level in $T\times I$ of $W_+$ and $V_-$ and the attaching disks for the $1$-handles go into the $1$-levels in these products.
	In order to proceed the amalgamation, isotope the attaching disks of both $1$-handles so that they can be disjoint after the projection into $T$ (see Figure \ref{fig-amalgamation}.)
	Note that the projection function may also be changed during this isopoty so that the projection images of these attaching disks do not intersect one another.
	Let us consider tho curves $C_E$ and $C_D$ in $T$ such that $C_E$ ($C_D$ resp.) bounds a disk in $T$ and this disk contains the projection images of the attaching disks for $W_+$ ($V_-$ resp.) in its interior.
	Make sure that the disk for $C_E$ is disjoint from that of $C_D$.
	By using the projection function, we can find the corresponding curve $\bar{C}_E$ in $\partial_+ W_+$ ($\bar{C}_D$ in $\partial_+ V_-$ resp.).
	Consider a disk $E\subset W_+$ ($D\subset V_-$ resp.) such that $\partial E = \bar{C}_E$ ($\partial D=\bar{C}_D$ resp.) like Figure \ref{fig-amalgamation}.
	Obviously, $D$ and $E$ are separating in their compression bodies.
	After amalgamation, $D$ and $E$ are disjoint separating disks in their handlebodies and the boundary of one is not isotopic to that of the other in the resulting Heegaard surface.\\

\begin{figure}
	\includegraphics[width=12cm]{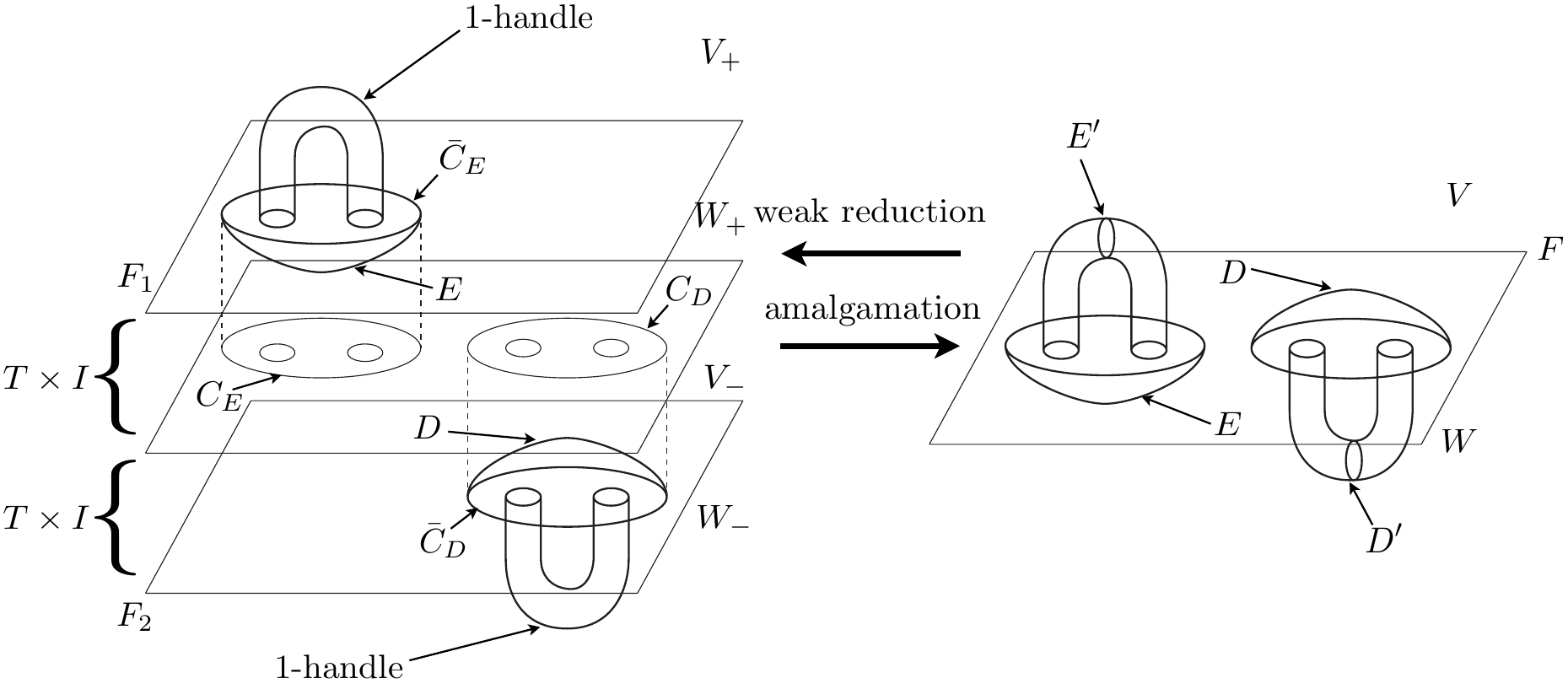}
	\caption{In the amalgamation case, we can find a weak reducing pair where both disks are separating.\label{fig-amalgamation}}
\end{figure}

$(\Rightarrow)$ 
	Suppose that there is a weak reducing pair $(D,E)$ where both disks are separating and the boundaries of both disks are not isotopic in $F$.
	Let $D\subset V$ and $E\subset W$.
	Now $F_{DE}$ consists of three tori $F'_{DE}$, $F''_{DE}$, and $F'''_{DE}$, where $F'_{DE}$ ($F'''_{DE}$ resp.) bounds a solid torus $V'$ in $V$ ($W'$ in $W$ resp.). 
	Find a meridian disk $D'$ ($E'$ resp.) in $V'$ ($W'$ resp.) such that $D'$ ($E'$ resp.) is disjoint from $D$ ($E$ resp.).
	Let $T$ be the surface  $F_{D'E'}$.
	Then, $F$ is obtained from $T$ by attaching boundaries of two $1$-handles whose cocores are $D'$ and $E'$ and removing the interiors of attaching disks of these handles.
	Now we get the right figure of Figure \ref{fig-amalgamation}.
	Since $V'\cap W'=\emptyset$, $(D', E')$ is a weak reducing pair.
	Therefore, if we perform the weak reduction by $(D',E')$, then we get the generalized Heegaard splitting where both thick surfaces are of genus two and the unique thin surface is torus.
\end{proof}

\begin{lemma} \label{lemma-differentlabel}
	Assume $M$, $H$, and $F$ as in Lemma \ref{lemma7} and consider $\WR$.
	If there is an edge labelled D (E resp.) in $\WR$ whose E-disk (D-disk resp.) is non-separating in its compression body, then $\WR$ must have a line segment of length two whose edges are labelled differently.
\end{lemma}

\begin{proof}
	Without loss of generality, suppose that there is an edge $(D_0, E)-_{e}(D_1, E)$ in $\WR$, $\partial D_0$ is separating, $\partial D_1$ is non-separating, and $\partial E$ is non-separating in $F$ by Corollary \ref{corollary-pants} and the assumption of this lemma.

	Let $F'_{D_0}$ be the genus two component of $F_{D_0}$.
	We can check that $\partial E\subset F'_{D_0}$ and $\partial D_1\cap F'_{D_0}=\emptyset$ by Corollary \ref{corollary-pants}.
	If we compress $F'_{D_0}$ along $E$ again, then $(F'_{D_0})_{E}$ is a torus.
	Let $E'$ and $E''$ be two parallel copies of $E_0$ and $D_0'$ be a parallel copy of $D_0$  in $(F'_{D_0})_E$.
	If we using the band sum arguments by using an arc connecting $E'$ and $E''$ missing $\partial E'$, $\partial E''$, and $\partial D_0'$ on its interior as in the proof of Lemma \ref{lemma-DorE}, we get a family $\{E_\alpha\}$ of infinitely many essential separating disks in $W$ which miss $E$ and $D_0$.
	Choose a disk $E_\alpha$ from $\{E_\alpha\}$, and make an edge $(D_0, E_\alpha)-_{f}(D_0,E)$ labelled E.
	Therefore,  $(D_0, E_\alpha)-_{f}(D_0, E)-_{e}(D_1, E)$ is the wanted line segment of length two whose edges are labelled differently.
\end{proof}

\begin{corollary} \label{corollary-differentlabel}
	Assume $M$, $H$, and $F$ as in Lemma \ref{lemma7}, consider $\WR$, and add the assumption that $M$ is closed.
	If $\dim(\WR)=1$, then $\WR$ must have a line segment of length two whose edges are labelled differently.
\end{corollary}

\begin{proof}
	Without loss of generality, suppose that there is an edge $(D_0, E)-_{e}(D_1, E)$ in $\WR$, $\partial D_0$ is separating, and $\partial D_1$ is non-separating in $F$ by Corollary \ref{corollary-pants}.
	If $\partial E$ is non-separating in $F$, then the proof ends by Lemma \ref{lemma-differentlabel}.
	Therefore, assume that $\partial E$ is separating in $F$, i.e. $E$ is separating in $W$.
	Hence, $E$ cuts $W$ into a solid torus $W'$ and a genus two handlebody, where $W'$ is bounded by $\tilde{F}\cup E$ in $W$. 
	($\tilde{F}$ is a once-punctured torus which $\partial E$ cut off from  $F$.) 
	Let $E_m$ be a meridian disk of $W'$. 
	Isotope $E_m$ in $W'$ so that $\partial E_m$ is contained in the interior of $\tilde{F}$.
	Let $\bar{F}$ be the once-punctured genus two component of $F-\partial D_0$.
	Since $\partial E\subset \bar{F}$ by Corollary \ref{corollary-pants} and $\partial E$ cuts $\bar{F}$ into a twice-punctured torus and a once-punctured torus (otherwise, $\partial E$ is isotopic to $\partial D_0$ in $\bar{F}$), we get either $\tilde{F}\subset \bar{F}$ or $\tilde{F}$ is of genus two.
	The latter case contradicts that $\tilde{F}$ is a once-punctured torus.
	Hence, $\partial D_0 \cap \partial E_m=\emptyset$, i.e. we get an edge $(D_0, E_m)-_{f}(D_0,E)$ labelled E. 
	Therefore, $(D_0, E_m)-_{f}(D_0, E)-_{e}(D_1, E)$ is the wanted line segment of length two whose edges are labelled differently. 
\end{proof}

Let us prove Corollary \ref{corollary-main}.
If $M$ is not an amalgamation of two genus two splitting along a torus, then we cannot choose a weak reducing pair whose disks are separating in their hadlebodies by Lemma \ref{lemma-corollary}.
(Since $M$ is irreducible and the splitting is unstabilized, the splitting is irreducible, i.e. we can get rid of the possibilities of weak reducing pairs whose disks have isotopic boundaries.)
Hence, there is no weak reducing pair such that each of both disks cuts off a solid torus in its compression body since the genus of the splitting is three.
Therefore, if we use Theorem \ref{theorem-main}, then the proof ends.\\

In addition, we can induce some equivalent conditions for $\dim(\WR)=0$ when $M$ is closed. 

\begin{corollary}\label{corollary-last}
	Let $M$ be a closed orientable irreducible $3$-manifold and $H=(V,W;F)$ be an unstabilized weakly reducible genus three Heegaard splitting.
	Then, the following three statements are equivalent.
	\begin{enumerate}
		\item \label{cm1} There is no weak reducing pair such that each of both disks cuts off a solid torus in its handlebody.
		\item \label{cm2} $H$ cannot be represented as an amalgamation of genus two splittings along a torus.
		\item \label{cm3}$\dim(\WR)=0$.
	\end{enumerate}
\end{corollary}

\begin{proof}
	Since $M$ is irreducible, (\ref{cm1}) and (\ref{cm2}) are equivalent by Lemma \ref{lemma-corollary}.
	Therefore we will prove that (\ref{cm1}) and (\ref{cm3}) are equivalent.

	(\ref{cm1})$\Rightarrow$(\ref{cm3})
	Suppose that $\dim(\WR)=1$, i.e. the negation of $\dim(\WR)=0$ by Lemma \ref{lemma7}.
	Then, $\WR$ must have a line segment of length two whose edges are labelled differently by Corollary \ref{corollary-differentlabel}.
	This gives the four disks $\bar{D}$, $\tilde{D}\subset V$ and $\bar{E}$, $\tilde{E}\subset W$ satisfying the four conditions of Corollary \ref{corollary-difffourcurves}.
	Let $\bar{D}$ and $\bar{E}$ be the two separating disks among them.
	Since we can assume $\bar{D}\cap \bar{E}=\emptyset$, we get a weak reducing pair of separating disks $(\bar{D}, \bar{E})$.
	By the assumption that $M$ is closed, $\bar{D}$ ($\bar{E}$ resp.) cuts off a solid torus from $V$ ($W$ resp.).

	(\ref{cm3})$\Rightarrow$(\ref{cm1})
	Suppose that $\dim(\WR)=0$.
	Suppose that there is weak reducing pair $(D, E)$ such that each of both disks cuts off a solid torus in its handlebody.
	Let the solid torus which $D$ cuts off from $V$ be $V'$ and $D_m$ be a meridian disk of $V'$.
	If we isotope $D_m$ in $V'$ so that $\partial D_m$ misses $D$, then we get a $1$-simplex $(D, E)-(D_m, E)$ in $\WR$, i.e $\dim(\WR)=1$ and this gives a contradiction.
	Therefore, There is no weak reducing pair such that each of both disks cuts off a solid torus in its handlebody.
\end{proof}

\section{Examples of Theorem \ref{theorem-main}\label{section-examples}}

\subsection{The three-torus $T^3$}
It is well known that $T^3 = (\text{torus})\times S^1$ has the unique minimal genus three Heegaard splitting (see Theorem 4.2 of \cite{FH} and Theorem 5.7 of \cite{Schultens1}.)

Let $C$ be a cube $\{(x,y,z)\in\mathbb{R}^3|-1\leq x,y,z\leq 1\}$ (see (a) of Figure \ref{fig-t3}.)
If we identify the three pairs of opposite faces of $C$, we get a 3-torus $M = (\text{torus})\times S^1$.
Let $q : C\to M$ be the quotient map.
An image by $q$ of a tubular neighborhood of union of three axis in $C$ is a genus three handlebody $V$.
Let $\overline{M-V}$ be $W$.
Then it is easy to see that $W$ is also a genus three handlebody.
This is the standard Heegaard splitting of genus three for $T^3$.

\begin{figure}
	\includegraphics[width=12cm]{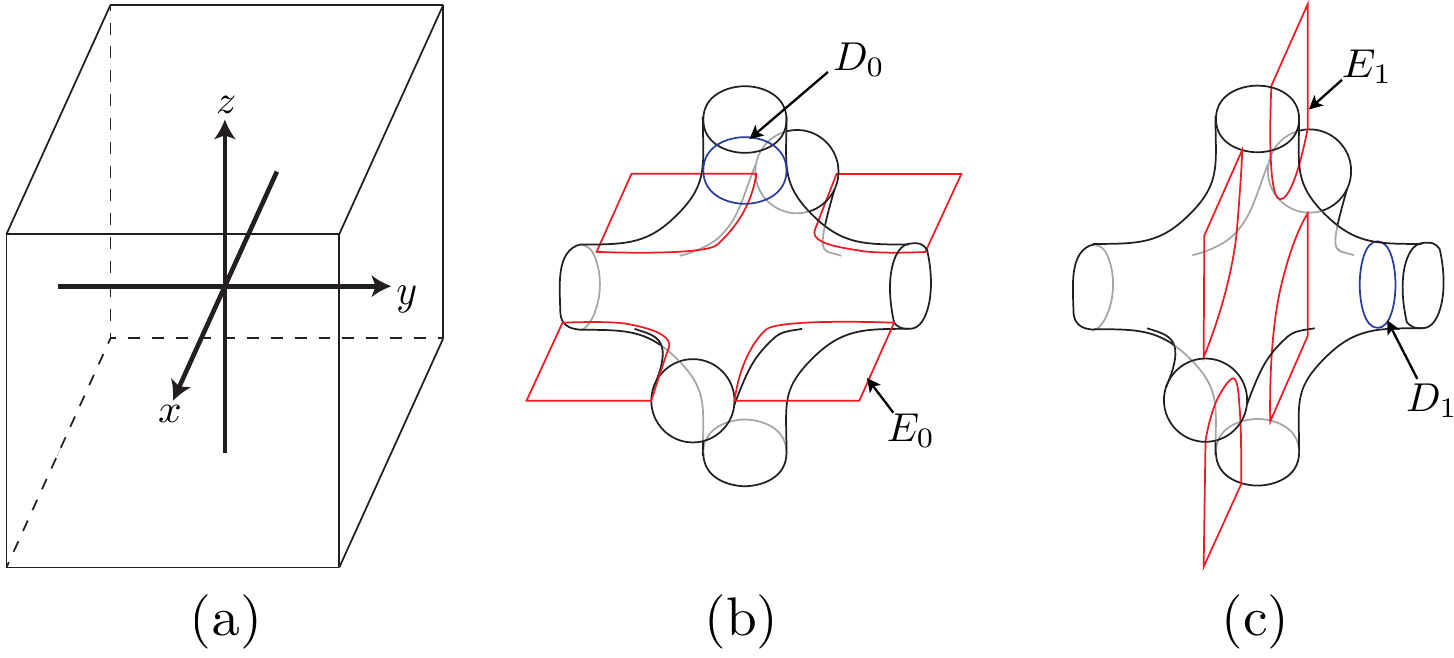}
	\caption{The standard Heegaard splitting of genus three of $T^3$, and its compressing disks which realize the critical Heegaard splitting.\label{fig-t3}}
\end{figure}

We will prove that this splitting is critical by using Theorem \ref{theorem-main}.
If there is a weak reducing pair for the standard splitting of genus three for $T^3$, then both disks are non-separating in their handlebodies (see Lemma 6 of \cite{JJ}.)
Therefore, there is no choice of a weak reducing pair such that each of both disks cuts off a solid torus in its handlebody.
By the existence of the weak reducing pairs $(D_0,E_0)$ and $(D_1, E_1)$ as in Figure \ref{fig-t3}, we can induce that this splitting is critical by Theorem \ref{theorem-main}.
This is an example of $\dim(\WR)=0$ (see Corollary \ref{corollary-last}.)
Also we can check that the standard splitting of genus three for $T^3$ cannot be represented by an amalgamation of two genus two Heegaard splittings along a torus by Corollary \ref{corollary-last}.

Let us check Corollary \ref{corollary-main-2} for this example.
Let $S$ be the horizontal torus in $T^3$ in the left of Figure \ref{fig-t3-incomp} and $F$ be the Heegaard surface of this splitting.
Since the fibers of the torus bundle $T^3=T^2\times S^1$ is incompressible in $T^3$, $S$ is incompressible.
Here, we can check that $S\cap F$ is inessential in $S$.
Let $c$ be an essential closed curve in $S$ as in the left of Figure \ref{fig-t3-incomp}.
If we push $S$ upward along $c$ until we get the right of Figure \ref{fig-t3-incomp}, then we get a torus $S'$ isotopic to $S$.
$S'\cap F$ consists of two circles, where each of both is depicted as a dotted curve in the right of Figure \ref{fig-t3-incomp}.
These curves are all essential in both $F$ and $S'$.
Therefore, $S'$ is the surface satisfying Corollary \ref{corollary-main-2}.

Note that there is another proof to show that this splitting is critical (see \cite{Lee}.)

\begin{figure}
	\includegraphics[width=9cm]{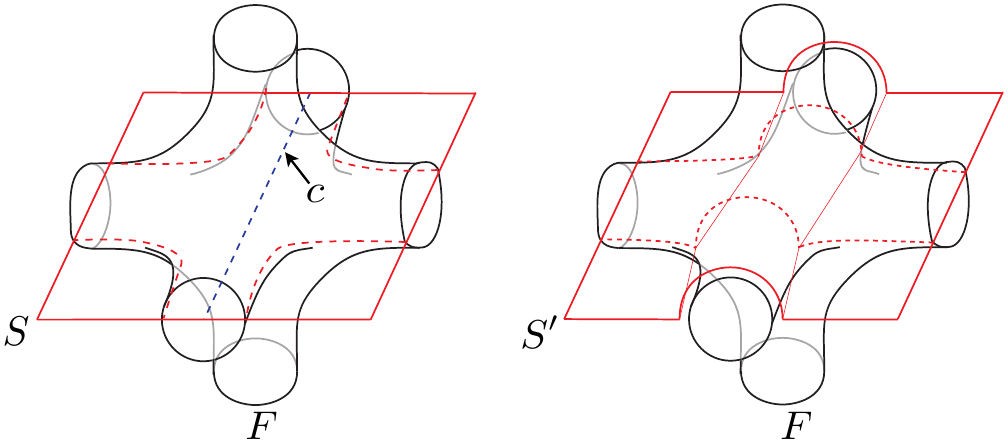}
	\caption{$S'$ is satisfying Corollary \ref{corollary-main-2}.\label{fig-t3-incomp}}
\end{figure}

\subsection{The three component chain exterior}
Let $M$ be the exterior of the three component chain pictured in Figure \ref{fig-three-chain}.
(The manifold $M$ may also be considered as $P\times S^1$, where $P$ is a pair of pants.)
Let $T_1$, $T_2$, and $T_3$ be the three boundary components of $\partial M$, and let us consider two arcs $a_1$ and $a_2$ such that $T_1$ and $T_2$ ($T_3$ and $T_2$ resp.) are connected by $a_1$ ($a_2$ resp.) as in Figure \ref{fig-three-chain}.
\begin{figure}
	\includegraphics[width=9cm]{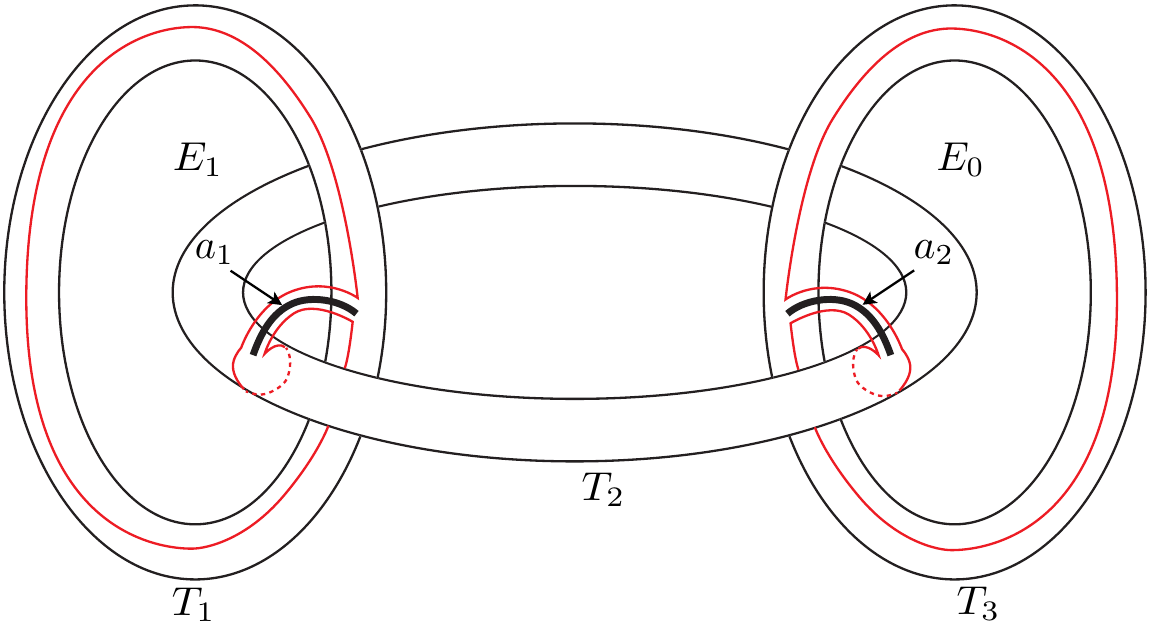}
	\caption{A tunnel system of the three component chain\label{fig-three-chain}}
\end{figure}
It is easy to show that $\{a_1, a_2\}$ is a tunnel system for $M$.
Therefore we get a Heegaard splitting $(V,W;F)$ of $M$, where $V$ is a genus three compression body whose minus boundary consists of $T_1$, $T_2$, and $T_3$ and $W$ is a genus three handlebody.
If the splitting is stabilized, then we can reduce the tunnel system.
But a compression body whose minus boundary consists of three tori must be of at least genus three, we cannot reduce this tunnel system, i.e. the splitting is unstabilized.
Let $D_0$ ($D_1$ resp.) be the cocore of $a_1$ ($a_2$ resp.) and $E_0$ ($E_1$ resp.) be the compressing disk for $W$ depicted as in Figure \ref{fig-three-chain}. 
Here, $(D_0, E_0)$ and $(D_1, E_1)$ are weak reducing pairs.
Moreover, $D_0\cap E_1\neq \emptyset$ and $D_1\cap E_0\neq \emptyset$. 
If there is an essential separating disk in $V$, then 
it cannot cut off a compression body with empty minus boundary from $V$.
That is, there is no weak reducing pair such that each of both disks cuts off a solid torus in its compression body.
Therefore, this splitting is critical by Theorem \ref{theorem-main}.

Now we claim that $\dim(\WR)=1$. 
Let $\bar{E}_0$ and $\tilde{E}_0$ be two parallel copies of $E_0$ as in Figure \ref{fig-three-chain-2}.
\begin{figure}
	\includegraphics[width=9cm]{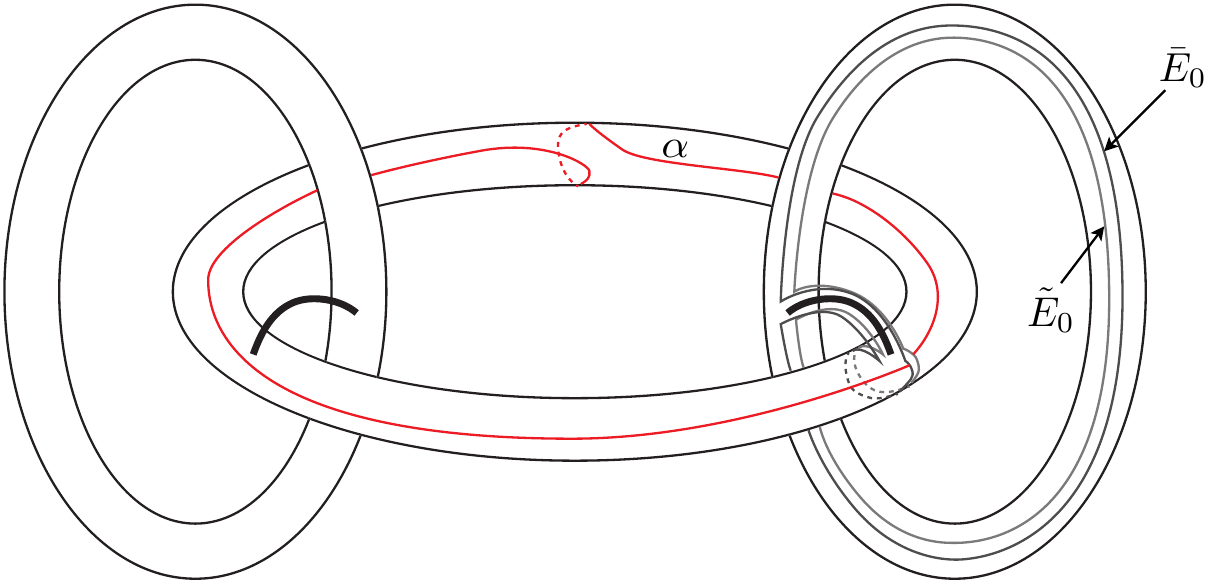}
	\caption{We can find the disk $E'$ by the band sum of two parallel copies of $E_0$\label{fig-three-chain-2}}
\end{figure}
If we consider the band sum of $\bar{E}_0$ and $\tilde{E}_0$ by the arc $\alpha$ connecting $\bar{E}_0$ and $\tilde{E}_0$ as in Figure \ref{fig-three-chain-2}, then we get an essential separating disk $E'$ in $W$.
Since $E'\cap D_0=\emptyset$, $(D_0, E_0)-(D_0, E')$ is a $1$-simplex in $\WR$.
This example means that the first statement and the third statement of Corollary \ref{corollary-last} are not equivalent for manifolds with non-empty boundary.

Note that E. Sedgwick proved that this splitting is not of minimal genus in \cite{Sedgwick}.

\section*{Acknowledgments}
The author is grateful to Prof. Jung-Hoon Lee for carefully reading the paper.
The author was supported by the National Research Foundation of Korea grant funded by
the Korean Government(Ministry of Education, Science and Technology)[NRF-2011-355-C00014].

\bibliographystyle{agsm}

\begin{thebibliography}{10}

\bibitem{Bachman1} D. Bachman, \emph{Critical Heegaard surfaces}. Trans. Amer. Math. Soc. \textbf{354} (2002), no. 10, 4015--4042 (electronic). 

\bibitem{Bachman2} D. Bachman, \emph{Connected sums of unstabilized Heegaard splittings are unstabilized}, Geom. Topol. \textbf{12} (2008), no. 4,  2327--2378.

\bibitem{Bachman3} D. Bachman, \emph{Topological index theory for surfaces in 3-manifolds}, Geom. Topol.  \textbf{14} (2010), no. 1, 585--609.

\bibitem{CassonGordon} A. J. Casson and C. McA. Gordon, \emph{Reducing Heegaard splittings}, Topology Appl. \textbf{27} (1987), no. 3, 275--283. 

\bibitem{FH} C. Frohman and J. Hass, \emph{Unstable minimal surfaces and Heegaard splittings}, 
Invent. Math. \textbf{95} (1989), no. 3, 529--540. 

\bibitem{JJ} J. Johnson, \emph{Automorphisms of the three-torus preserving a genus-three heegaard splitting}, 
Pacific J. Math. \textbf{253} (2011),  no. 1, 75--94.

\bibitem{Lee} J. Lee,
\emph{Examples of unstabilized critical Heegaard surfaces}, arXiv:1005.0974.

\bibitem{Moriah} Y. Moriah, \emph{On boundary primitive manifolds and a theorem of Casson-Gordon}, Topology Appl. \textbf{125} (2002), no. 3, 571--579.  

\bibitem{SaitoScharlemannSchultens} M. Saito, M. Scharlemann and J. Schultens, \emph{Lecture notes on generalized Heegaard splittings}, arXiv:math/0504167v1

\bibitem{ScharlemannSchultens1} M. Scharlemann and J. Schultens,
\emph{The tunnel number of the sum of $n$ knots is at least $n$}, 
Topology \textbf{38} (1999), no. 2, 265--270. 

\bibitem{Schultens1} J. Schultens, \emph{The classification of Heegaard splittings for (compact orientable surface)$\times S^1$}, Proc. London Math. Soc. (3) \textbf{67} (1993), no. 2, 425--448.

\bibitem{Schultens2} J. Schultens, \emph{Additivity of tunnel number for small knots}, Comment. Math. Helv. \textbf{75} (2000), no. 3, 353--367.

\bibitem{Sedgwick}  E. Sedgwick, \emph{Genus two 3-manifolds are built from handle number one pieces}, Algebr. Geom. Topol. \textbf{1} (2001), 763--790.

\bibitem{Thompson} A. Thompson, \emph{The disjoint curve property and genus 2 manifolds}, Topology Appl. \textbf{97} (1999) 273--279.

\bibitem{Waldhausen} F. Waldhausen, \emph{Heegaard-Zerlegungen der 3-sph\"{a}re}, Topology, \textbf{7} (1968), 195--203.
\end{thebibliography}

\end{document}